\documentclass[11pt,reqno]{amsart}
\usepackage{graphicx}
\usepackage{color}
\usepackage{amsmath,amssymb,amsthm,amsfonts}
\usepackage{mathdots}
\usepackage{graphicx}
\usepackage{color}
\usepackage{multicol}
\def\R{\mathbb{R}}
\makeatletter

\newcommand{\Rmnum}[1]{\expandafter\@slowromancap\romannumeral #1@}
\makeatother
\newcommand{\D}{\displaystyle}
\newtheorem{thm}{Theorem}[section]

\newtheorem{lemma}[thm]{Lemma}
\newtheorem{remark}{Remark}[section]
\newtheorem{theorem}[thm]{Theorem}
\newtheorem{proposition}[thm]{Proposition}

\usepackage[numbers,sort&compress]{natbib}
\begin{document}
\author{Hai-Yang Jin}
\address{Department of Mathematics, South China University of Technology, Guangzhou 510640, China}
\email{mahyjin@scut.edu.cn}

\author{Zhian Wang}
\address{Department of Applied Mathematics, Hong Kong Polytechnic University, Hung Hom, Hong Kong}
\email{mawza@polyu.edu.hk}

\title[Global Stability and spatio-temporal Patterns of Predator-Prey systems]{Global dynamics and spatio-temporal patterns of predator-prey systems with density-dependent motion}

\maketitle

\begin{quote}
\small{
{\bf Abstract}: In this paper,  we investigate the global boundedness, asymptotic stability and pattern formation of predator-prey systems with density-dependent preytaxis in a two-dimensional bounded domain with Neumann boundary conditions, where the coefficients of motility (diffusion) and mobility (preytaxis) of the predator are correlated through a prey density dependent motility function.  We establish the existence of classical solutions with uniform-in time bound and the global stability of the spatially homogeneous prey-only steady states and coexistence steady states under certain conditions on parameters by constructing Lyapunov functionals. With numerical simulations, we further demonstrate that spatially homogeneous time-periodic patterns, stationary spatially inhomogeneous patterns and chaotic spatio-temporal patterns are all possible for the parameters outside the stability regime. We also find from numerical simulations that the temporal dynamics between linearized system and nonlinear systems {\color{black}are} quite different, and the prey density-dependent motility function can trigger the pattern formation.\\


}
\end{quote}


\numberwithin{equation}{section}
\section{Introduction}
The foraging is the searching for wild food resources by hunting, fishing, consuming or the gathering of plant matter. It plays an important role in an organism's ability to survive and reproduce.
The nonrandom foraging strategies in the predator-prey dynamics, such as the area-restricted search, is often observed to result in  populations of predators moving (or flowing) toward regions of higher prey density (see \cite{Curio, Chesson, Murdoch, Okubo}). Such movement is referred to as preytaxis which has important roles in biological control or ecological balance such as regulating prey (pest) population to avoid incipient outbreaks of prey or forming large-scale aggregation for survival (cf. \cite{Sapoukhina, Murdoch, Grunbaum}).
To understand the dynamics of predator-prey systems with prey-taxis,  Karevia and Odell \cite{KO-1987} put individual foraging behaviors into a biased random walk model which, upon passage to a continuum limit, leads to the following preytaxis system  (see equations (55)-(56) in \cite{KO-1987}):
\begin{eqnarray}\label{pt1}
\left\{\begin{array}{lll}
\partial_t u = \nabla \cdot (d(v) \nabla u)- \nabla \cdot( u\chi(v) \nabla v)+\mathcal{F}(u,v),\\
\partial_t v = D\Delta v+\mathcal{G}(u,v),
\end{array}\right.
\end{eqnarray}
where $u=u(x,t)$ and $v=v(x,t)$ denote the {\color{black} population} density of predators and preys at position $x$ and time $t$, respectively,  and  $D>0$ is a constant denoting the diffusivity of preys.
The term $\nabla \cdot (d(v) \nabla u)$ describes the diffusion (motility) of predators with coefficient $d(v)$, and $-\nabla \cdot( u\chi(v) \nabla v)$ accounts for the prey-taxis (mobility) with coefficient $\chi(v)$, where both motility and mobility coefficients are related to individual foraging behaviors.  The source terms $\mathcal{F}(u,v)$ and $\mathcal{G}(u,v)$ represent the predator-prey interactions to be discussed below. By fitting the abstract model (\ref{pt1})  to field experiment data of area-restricted search behavior exhibited by individual ladybugs (predators) and aphids (prey) with appropriate predator-prey interactions  (see \cite{KO-1987}), Karevia and Odell showed that the area-restricted non-random foraging  yield heterogeneous  aggregative patterns observed in the field experiment.

In a special case $\chi(v)=-d'(v)$, the system \eqref{pt1} becomes
\begin{eqnarray*}\label{pt1n}
\left\{\begin{array}{lll}
\partial_t u = \Delta(d(v)u)+\mathcal{F}(u,v),\\
\partial_t v = D\Delta v+\mathcal{G}(u,v),
\end{array}\right.
\end{eqnarray*}
where the diffusion term $\Delta(d(v)u)$ with $d'(v)<0$ has been interpreted as ``density-suppressed motility''  in \cite{Fu, Liu}(see more in \cite{JKW-SIAP-2018, Smith}), and $d(v)$ is called the motility function. This means that the predator will reduce its motility when encountering the prey, which is a rather reasonable assumption and has very sound applications in the predator-prey systems.

As mentioned in \cite{KO-1987},  the model (\ref{pt1}) was tailored to study the non-random foraging behavior (or prey-taxis) not only for ladybugs and aphids, but also for general organisms living in the predator-prey system.
Formally the model (\ref{pt1}) can be regarded as a variant of the Keller-Segel chemotaxis model \cite{KS1}, where $u(x,t)$ denotes the cell density and $v(x,t)$ the chemical concentration. However the prey-taxis model (\ref{pt1}) has two striking features that the Keller-Segel  models have not considered yet. First the model (\ref{pt1}) characterizes the non-random population dispersal and aggregation (i.e. both diffusion  and prey-taxis coefficients depend on the prey density). Second, the source terms in (\ref{pt1}) have the inter-specific interactions. These two features distinguish the prey-taxis model from the Keller-Segel type chemotaxis models.

Ecological/biological interactions can be defined as either intra-specific or inter-specific. The former occurs between individuals of the same species, while the later between two or more species.  There are three types of basic interspecific interactions (see \cite{cosner, Jungel, LHL-2009}): predator-prey, competition and mutualism, which can be encapsulated in $\mathcal{F}(u,v)$ and $\mathcal{G}(u,v)$ by the following typical form:
\begin{equation}\label{inter2}
\mathcal{F}(u,v)=c_1 u F(v)+h(u), \ \ \mathcal{G}(u,v)=f(v)-c_2 u F(v)
\end{equation}
where $h(u)$ and $f(v)$ are functions representing the intra-specific interactions of predators and preys, respectively. Parameters $c_1, c_2\in\R$ denote the coefficients of inter-specific interactions between predators and preys, where $F(v)$ is commonly called the functional response function fulfilling
$F(0)=0, \ F'(v)>0$. This paper is interested in the predator-prey interaction where $c_1>0, c_2>0$.
The function forms given in (\ref{inter2}) have represented most of ecologically meaningful examples of predator-prey population interactions used in the literature by assigning appropriate expressions to $h(u)$, $f(v)$ and $F(v)$. Typically the predator kinetics $h(u)$ may include density dependent death $h(u)=-u(\theta+\alpha u), \theta>0, \alpha \geq0$ , the prey kinetics $f(v)$ could be linear, logistic or Allee effect (bistable) type, and $F(v)$ may be of Lotka-Volterra type \cite{Lotka,Volterra} or Holling's type \cite{Holling2}. We refer the readers to the excellent surveys \cite{Murdoch-Briggs, Turchin} for an exhaustive list of $h(u), f(v)$ and $F(v)$.
Hence in this paper we consider the following system
\begin{equation}\label{1-1}
\begin{cases}
u_t=\nabla\cdot (d(v)\nabla u)- \nabla\cdot( u\chi(v)\nabla v)+\gamma uF(v)-\theta u-\alpha u^2, &x\in \Omega, ~~t>0,\\
 v_t=D\Delta v- u F(v)+ f(v),& x\in \Omega, ~~t>0,\\
 \frac{\partial u}{\partial\nu}=\frac{\partial v}{\partial\nu}=0,& x\in \partial\Omega, ~~t>0,\\
  u(x,0)=u_0(x),v(x,0)=v_0(x),&x\in\Omega,\\
 \end{cases}
\end{equation}
where $D,\gamma>0,\theta>0, \alpha\geq 0$ are  constants, and $d(v),\chi(v), F(v)$ and $f(v)$ satisfy the following conditions:
\begin{itemize}
\item[(H1)] $d(v),\chi (v) \in C^2([0,\infty)), d(v)>0, {\color{black}\chi(v)\geq 0}$ and ${\color{black}d'(v)\leq 0}$ on $[0,\infty)$.
\item[(H2)] $F(v)\in {\color{black}C^1([0,\infty))}, ~F(0)=0,~F(v)>0 ~\mathrm{in}~(0,\infty)~\mathrm{and}~F'(v)>0$.
\item[(H3)] $f: [0,\infty)\to\R$ is in $C^1$ with $f(0)=0$, and there exist two constants $\mu, K>0$ such that  $f(v)\leq \mu v$ for any $v\geq0$, $f(K)=0$ and $f(v)<0$ for all $v>K$.
\end{itemize}
We remark that the above assumptions for $F(v)$ and $f(v)$ have covered a large class of interesting and meaningful examples encountered in the literature as mentioned above. Our first result on the global boundedness of solutions of \eqref{1-1} is the following.


\begin{theorem}[Global boundedness]\label{GB}
Assume $(u_0,v_0)\in [W^{1,p}(\Omega)]^2$ with $p>2$ and $u_0, v_0 \geq 0(\not \equiv 0)$. Let $\Omega \subset \R^2$ be a bounded domain with smooth boundary and the hypotheses (H1)-(H3) hold. {\color{black} If $\alpha >0$ or $\chi(v)=-d'(v)$},
then there is a unique classical solution $(u,v) \in [C([0,\infty)\times\bar{\Omega}) \cap C^{2,1}((0,\infty)\times\bar{\Omega})]^2$ solving the problem (\ref{1-1}). Moreover there is constant $C>0$ independent of $t$ such that
\begin{equation*}\label{uf}
\|u(\cdot,t)\|_{L^\infty}+\|v(\cdot,t)\|_{W^{1, \infty}}\leq C,
\end{equation*}
where in particular $0\leq v\leq K_0$ with
\begin{equation}\label{K0}
K_0:=\max\{\|v_0\|_{L^\infty},K\}.
\end{equation}
\end{theorem}
Next, we will study the large time behavior of solutions. One can easily compute that  the system (\ref{1-1}) has three  homogeneous steady states $(u_s, v_s)$:
 \begin{equation*}\label{ss1}
 (u_s, v_s)=
\begin{cases}
(0,0) \ \mathrm{or} \ (0,K), & \mathrm{if} \ \gamma F(K) \leq \theta{\color{black},}\\
(0,0) \ \mathrm{or} \ (0,K) \ \mathrm{or} \ (u_*,v_*), & \mathrm{if} \ \gamma F(K) > \theta\\
\end{cases}
\end{equation*}
with $u_*, v_*>0$ determined by the following algebraic equations:
\begin{equation}\label{ss2}
u_*=\frac{f(v_*)}{F(v_*)}, \ \gamma F(v_*)=\theta +\alpha u_*,
\end{equation}
where $(0,0)$ is the extinction steady state, $(0,K)$ is the prey-only steady state and $(u_*, v_*)$ is the coexistence steady state. As  in \cite{JW-JDE-2017}, for the global stability, along with the hypotheses (H1)-(H3), we need another condition for the following compound {\color{black}function}:
\begin{equation*}\label{phi}
\phi(v)=\frac{f(v)}{F(v)}
\end{equation*}
as follows:
\medskip

(H4) The function $\phi(v)$ is continuously differentiable on $(0, \infty)$, $\phi(0)=\lim\limits_{v \to 0} \phi(v)>0$ and $\phi'(v)<0$ for any $v\geq 0$.
\medskip

\bigbreak

Then the global stability results are given as follows:

\begin{theorem}[Global stability]\label{WP}
Let the hypotheses (H1)-(H4) and assumptions in Theorem $\ref{GB}$ hold. Then the solution $(u,v)$ obtained in Theorem $\ref{GB}$ has the following properties:
 \begin{enumerate}
\item If  the parameters $\theta,\gamma,K$ satisfy
$
\gamma F(K)\leq \theta
$
where ``$=$'' holds iff $\alpha>0$,
 then
 $$\|u\|_{L^\infty}+\|v-K\|_{L^\infty} \to 0 \ \text{as}\  t\to \infty$$
 exponentially if $\gamma F(K)< \theta$ or algebraically if $\gamma F(K)= \theta, \alpha>0$.

\item If the  parameters $\theta,\gamma, K$  satisfy
$\gamma F(K)>\theta$ and
\begin{equation*}\label{sc}
D \geq\max\limits_{0\leq v\leq K_0}{\color{black} \frac{ u_* |F(v)|^2|\chi(v)|^2}{4 \gamma F(v_*)F'(v)d(v)}},
\end{equation*}
then
$$ \|u-u_*\|_{L^\infty}+\|v-v_*\|_{L^\infty} \to 0 \ \text{as}\  t\to \infty,$$
where the convergence is exponential if $\alpha>0$.
 \end{enumerate}
\end{theorem}

{\color{black} We remark if $d(v)$ and $\chi(v)$ are constant, {the system (1.3) has been studied from various aspects (cf. \cite{ABN-2008, Chakraborty, Lee-Hillen-Lewis1, LHL-2009, Tao-2010, WSW-2015}), and in particular} the global existence and stability of solutions have been established in a previous work by the authors in \cite{JW-JDE-2017}. The results Theorem \ref{GB} and \ref{WP} extend the results of \cite{JW-JDE-2017} to nonconstant $d(v)$ and $\chi(v)$. Moreover the proof of Theorem \ref{WP} uses a similar idea (method of Lyaunoval functional) as in \cite{JW-JDE-2017}.  Except these analogies, we would like to stress some essential differences between the present paper and \cite{JW-JDE-2017} below.
\begin{itemize}
\item The method used in \cite{JW-JDE-2017} to prove the global existence of solutions was based on {\it a priori} estimate for the energy functional $\int_\Omega u\ln udx$ to attain the $L^2$-estimate of solutions. However such {\it a priori} estimates is attainable only for case where
      the motility function $d(v)$ is constant. Hence the method of \cite{JW-JDE-2017} is inapplicable to the model \eqref{1-1}. In this paper, we estimate $L^2$-norm of solutions directly to obtain the global existence of solutions (see section 3). With this new direct $L^2$-estimate method, the concavity of $F(v)$ required in \cite{JW-JDE-2017} is no longer needed. That is we not only use a different method to prove the global boundedness of solutions but also remove the condition $F''(v)\leq 0$ imposed in \cite{JW-JDE-2017}.
\item If $d(v)$ and $\chi(v)$ are non-constant as considered in this paper,  we find that the system (\ref{1-1}) can generate pattern formation as shown in section 5. This is different from \cite{JW-JDE-2017} where no pattern formation can be founded for constant $d(v)$ and $\chi(v)$. Our result indicates that the density-dependent motility is a trigger for pattern formation, which is a new finding.
\item In the proof of global stability shown in section 4, we need the estimate of $\|\nabla u\|_{L^4}$ which can be easily obtained for constant $d(v)$ and $\chi(v)$ but is not clear if $d(v)$ and $\chi(v)$ are not constant. Hence we present a new proof in Lemma 4.2.
\end{itemize}
}

In section 5, we shall detail the results of Theorem \ref{GB} and Theorem \ref{WP} for specific and often used forms of $F(v)$. In Theorem \ref{WP}  the conditions on parameter values are identified to ensure the global stability of homogeneous steady states. However it is unknown whether non-homogeneous (i.e., non-constant) steady states exists outside the stability regimes found in Theorem \ref{WP}. In the final section 5, we shall use linear stability analysis to find the conditions on parameters for the instability of equilibria of (\ref{1-1}) and then perform numerical simulations to illustrate that indeed spatially inhomogeneous patterns and time-periodic patterns can be found under certain conditions. We also demonstrate that the nonconstant motility function $d(v)$ plays an important role in generating the pattern formation.

\section{Local existence and Preliminaries}
In what follows, we shall use $c_i$ or $C_i (i=1,2,3, \cdots$) to denote a generic constant which may vary in the context.  We first state the existence of local-in-time classical solutions of system \eqref{1-1} by using the abstract theory of quasilinear parabolic systems in \cite{A-Book-1993}.
\begin{lemma}[Local existence]\label{LS}
Let $\Omega $ be a bounded domain in $\R^2$ with smooth boundary and the hypotheses (H1)-(H3) hold. Assume $(u_0,v_0)\in [W^{1,p}(\Omega)]^2$ with
$u_0, v_0 \geq 0(\not \equiv 0)$ and $p>2$.  Then there exists {\color{black}$T_{max}\in(0,\infty]$} such that the  problem (\ref{1-1}) has a unique classical solution  $(u,v) \in [C([0,T_{max})\times\bar{\Omega}) \cap C^{2,1}((0,T_{max})\times\bar{\Omega})]^2$ satisfying $u, v\geq 0$ for all $t>0$. Moreover, we have
{\color{black}\begin{equation}\label{bc}
either \  \ T_{max}=\infty, \ or\ \  \limsup\limits_{t\nearrow T_{max}}(\|u(\cdot,t)\|_{L^\infty}+\|v(\cdot,t)\|_{L^\infty})=\infty.
\end{equation}}
\end{lemma}
\begin{proof} We shall apply the theory developed by Amann \cite{A-Book-1993} to prove this lemma.
With $\omega=(u,v)$, the system \eqref{1-1} can be reformulated as
\begin{equation*}\label{RS}
\begin{cases}
\omega_t=\nabla \cdot (\mathbb{A}(\omega)\nabla \omega )+\Phi (\omega), &x\in \Omega,t>0,\\
\frac{\partial \omega }{\partial \nu}=0, &x\in\partial \Omega, t>0,\\
\omega(\cdot,0)=(u_0,v_0), & x\in \Omega,
\end{cases}
\end{equation*}
where
\begin{equation*}
\mathbb{A}(\omega)=\begin{bmatrix} d(v) & -\chi(v) u \\[1mm]0 &1 \end{bmatrix} \ \mathrm{and}\ \
\Phi(\omega)=\begin{bmatrix} \gamma uF(v)-\theta u-\alpha u^2 \\ f(v)-u F(v) \end{bmatrix}.
\end{equation*}
Since the given initial conditions satisfy  $0\leq (u_0,v_0)\in [W^{1,p}(\Omega)]^2$ with $p>2$ and hence
the matrix $\mathbb{A}(\omega)$ is positively definite at $t=0$.  Hence the system \eqref{1-1} is normally parabolic and the local existence of solutions follows from  \cite[Theorem 7.3]{A-DIE-1990}. That is there exists
a $T_{max}>0$ such that the system \eqref{1-1} admits a unique solution  $(u,v) \in [C([0,T_{max})\times\bar{\Omega}) \cap C^{2,1}((0,T_{max})\times\bar{\Omega})]^2$.

Next, we will use the maximum principle to prove that $u,v\geq 0$. To this end, we rewrite the first equation of \eqref{1-1} as
\begin{equation}\label{RS-1}
u_t=d(v)\Delta u+[d'(v) \nabla v-\chi(v)\nabla v]\cdot \nabla u-\chi'(v)u|\nabla v|^2-\chi(v)u\Delta v+\gamma u F(v)-\theta u-\alpha u^2.
\end{equation}
Then applying the strong maximum principle to \eqref{RS-1} with the Neumann boundary condition asserts that
$u>0$  for all $(x,t)\in \Omega \times (0,T_{max})$ due to $u_0\not \equiv 0$. Similarly, we can show $v>0$ for any  $(x,t)\in \Omega \times (0,T_{max})$ by applying the strong maximum principle to the second equation of system \eqref{1-1}. Since  $\mathbb{A}(\omega)$ is an upper triangular matrix,  the assertion \eqref{bc} follows from \cite[Theorem 5.2]{A-MZ-1989} directly.  Then the proof of Lemma \ref{LS} is completed.

\end{proof}
\begin{lemma}[\cite{JW-JDE-2017}]
Under the conditions in Lemma \ref{LS}, the solution $(u,v)$ of $\eqref{1-1}$ satisfies
 \begin{equation}\label{Lv}
 {0\leq v(x,t)\leq K_0},~~\mathrm{for~all}\ \ x\in\Omega \ \mathrm{and}\ t\in (0,T_{max}),
 \end{equation}
 where $K_0$ is defined by \eqref{K0}.
\end{lemma}
Next, we  present a basic boundedness property of the solutions to \eqref{1-1}.
\begin{lemma}\label{L1u} Let $(u,v)$ be a solution of system \eqref{1-1}. Then there exists  a constant $C>0$ such that
\begin{equation}\label{L1u-1}
\int_\Omega udx\leq C,\ \ \mathrm{for\ \ all }\ \ t\in(0,T_{max}).
\end{equation}
Moreover, if $\alpha>0$ or $\chi(v)=-d'(v)$, one has
\begin{equation}\label{2L1-2}
\int_t^{t+\tau} \int_\Omega u^2dxds\leq C, \ \ \mathrm{for\ \ all }\ \ t\in(0,\widetilde{T}_{max}),
\end{equation}
 where
\begin{equation}\label{DtT}
\tau:=\min \Big\{ 1,\frac{1}{2} T_{max}\Big\} \ \ \mathrm{and}\ \
\widetilde{T}_{max}:=\begin{cases}
T_{max}-\tau, \ \
&\mathrm{if}\ \ T_{max}<\infty,\\
\infty, &\mathrm{if}\ \ T_{max}=\infty.\\
\end{cases}
\end{equation}
\end{lemma}
\begin{proof}
Multiplying the second equation $\eqref{1-1}$ by $\gamma$ and adding the resulting equation into the first equation of $\eqref{1-1}$, then integrating the result over $\Omega\times(0,t)$, one has
\begin{equation*}
\begin{split}
\frac{d}{dt}\left(\int_\Omega udx+\gamma\int_\Omega vdx\right)+\int_\Omega u(\theta+\alpha u)dx
&=\gamma \int_\Omega f(v)dx, \\
\end{split}
\end{equation*}
which, along with the hypotheses (H2) and (H3) and the fact that $0\leq v\leq K_0$, gives
\begin{equation}\label{Eq-u1}
\begin{split}
\frac{d}{dt}\left(\int_\Omega udx+\gamma\int_\Omega vdx\right)+\theta\left(\int_\Omega udx+\gamma\int_\Omega vdx\right)+\alpha \int_\Omega u^2dx
&\leq(\gamma\mu+\theta\gamma)
 \int_\Omega vdx\\
 &\leq (\gamma\mu+\theta\gamma)K_0|\Omega|.
\end{split}
\end{equation}
Then the Gronwall's inequality applied to $\eqref{Eq-u1}$ yields  \eqref{L1u-1}.
If $\alpha>0$, then integrating  \eqref{Eq-u1} over $(t,t+\tau)$, we have \eqref{2L1-2} directly.

Next, we consider the case $\chi(v)=-d'(v)$.
In this case,  the first equation of system \eqref{1-1} can be written as
\begin{equation}\label{eq}
{\color{black}u_t}=\Delta(d(v)u)+\gamma uF(v)-\theta u-\alpha u^2.
\end{equation}
 Let $0<\delta<\frac{\theta}{d(0)}$ and let $\mathcal{A}$ denote the self-adjoint realization of
 $-\Delta+\delta$ under homogeneous Neumann boundary conditions in $L^2(\Omega)$. Since $\delta>0$, $\mathcal{A}$ possesses an order-preserving bounded inverse $\mathcal{A}^{-1}$ on $L^2(\Omega)$ and hence we can find a constant $c_1>0$ such that
 \begin{equation}\label{A1}
 \|\mathcal{A}^{-1}\psi\|_{L^2}\leq c_1\|\psi\|_{L^2} \ \ \ \mathrm{ for \  all}\ \psi\in L^2(\Omega)
 \end{equation}
 and
 \begin{equation}\label{A2}
  \|\mathcal{A}^{-\frac{1}{2}}\psi\|_{L^2}^2=\int_\Omega \psi\cdot  \mathcal{A}^{-1}\psi dx\leq c_1\|\psi\|_{L^2}^2   \ \ \ \mathrm{ for \  all}\ \psi\in L^2(\Omega).
 \end{equation}
 From the second equation of (\ref{1-1}) with \eqref{eq}, we have
 \begin{equation*}
 (u+\gamma v)_t=\Delta(d(v)u+D\gamma v)-\theta u-\alpha u^2+\gamma f(v),
 \end{equation*}
 which can be rewritten as
 \begin{equation}\label{A3}
 \begin{split}
  (u+\gamma v)_t+\mathcal{A}(d(v)u+D\gamma v)
  &=\delta [d(v)u+D\gamma v]-\theta u-\alpha u^2+\gamma f(v)\\
  &=(\delta d(v)-\theta-\alpha u) u+\gamma f(v)+D\delta\gamma v.
  \end{split}
 \end{equation}
 Noting the facts $0<\delta<\frac{\theta}{d(0)}$ and \eqref{Lv}, one can derive that
 \begin{equation}\label{A4}
 (\delta d(v)-\theta-\alpha u) u+\gamma f(v)+D\delta\gamma v\leq (\delta d(0)-\theta)+c_2\leq c_2.
 \end{equation}
Hence, multiplying \eqref{A3} by $\mathcal{A}^{-1} (u+\gamma v)\geq 0$, and using the fact \eqref{A4}, one has
\begin{equation*}
\frac{1}{2}\frac{d}{dt}\int_\Omega |\mathcal{A}^{-\frac{1}{2}} (u+\gamma v)|^2dx
+\int_\Omega (d(v) u+D\gamma v) (u+\gamma v)dx\leq c_2\int_\Omega \mathcal{A}^{-1}(u+\gamma v)dx,
\end{equation*}
which together with the fact $0<d(K_0)\leq d(v)$,  gives a constant  $c_3:=\min\{d(K_0),D\}$ such that
\begin{equation}\label{A5}
\frac{1}{2}\frac{d}{dt}\int_\Omega |\mathcal{A}^{-\frac{1}{2}} (u+\gamma v)|^2dx
+c_3\int_\Omega  (u+\gamma v)^2dx\leq c_2\int_\Omega \mathcal{A}^{-1}(u+\gamma v)dx.
\end{equation}
Using \eqref{A1} and \eqref{A2}, we can derive that
\begin{equation}\label{A6}
\begin{split}
\frac{c_3}{4c_1}\int_\Omega |\mathcal{A}^{-\frac{1}{2}} (u+\gamma v)|^2
dx+c_2\int_\Omega  \mathcal{A}^{-1}(u+\gamma v)dx
&\leq \frac{c_3}{4}\int_\Omega  (u+\gamma v)^2dx
+c_2|\Omega |^\frac{1}{2}\|u+\gamma v\|_{L^2}\\
&\leq \frac{c_3}{2}\int_\Omega  (u+\gamma v)^2dx+\frac{c_2^2|\Omega|}{c_3}.
\end{split}
\end{equation}
Adding \eqref{A6} and \eqref{A5}, and letting $y_1(t):=\int_\Omega |\mathcal{A}^{-\frac{1}{2}} (u+\gamma v)|^2$, one has
\begin{equation*}
y_1'(t)+\frac{c_3}{2c_1}y_1(t)+c_3\int_\Omega  (u+\gamma v)^2dx\leq \frac{2c_2^2|\Omega|}{c_3}.
\end{equation*}
Then one has  $y_1(t)\leq c_4$ and
\begin{equation*}\label{A7}
\int_{t}^{t+\tau}\int_\Omega u^2dxds\leq \int_{t}^{t+\tau}\int_\Omega  (u+\gamma v)^2dxds\leq \frac{y_1(t)}{c_3}+ \frac{2c_2^2|\Omega| \tau}{c_3^2}\leq \frac{c_4}{c_3}+ \frac{2c_2^2|\Omega| }{c_3^2}\ \  \mathrm{for \ \  all}\ \ t\in(0,\widetilde{T}_{max}),
\end{equation*}
which gives \eqref{2L1-2}. Then the proof of this lemma is completed.
 \end{proof}
Moreover, we can thereupon deduce the following result as a consequence of Lemma \ref{L1u}.
\begin{lemma} \label{2L1}
Let $(u,v)$ be the solution of system \eqref{1-1}, then there exists a constant $C>0$ independent of $t$ such that
\begin{equation}\label{2L1-1}
\|\nabla v\|_{L^2}\leq C,\ \ \mathrm{for\ \ all }\ \ t\in(0,T_{max})
\end{equation}
and
\begin{equation}\label{2L1-3}
\int_t^{t+\tau}\int_\Omega |\Delta v|^2dxds\leq C, \ \ \mathrm{for\ \ all }\ \ t\in(0,\widetilde{T}_{max}),
\end{equation}
where $\tau$ and $\widetilde{T}_{max}$ are defined by \eqref{DtT}.
\end{lemma}
\begin{proof}
Multiplying the second equation of system \eqref{1-1} by $-\Delta v$,  integrating the result by part and using the Cauchy-Schwarz inequality and {\color{black}the boundedness of $v$ in} \eqref{Lv}, we end up with
\begin{equation*}
\begin{split}
\frac{1}{2}\frac{d}{dt}\int_\Omega |\nabla v|^2dx+D\int_\Omega |\Delta v|^2dx
&=\int_\Omega u F(v)\Delta vdx-\int_\Omega f(v)\Delta v dx\\
&\leq \frac{D}{2}\int_\Omega |\Delta v|^2dx+\frac{1}{D}\int_\Omega  F^2(v)u^2dx+\frac{1}{D}\int_\Omega f^2(v)dx\\
&\leq \frac{D}{2}\int_\Omega|\Delta v|^2dx+ \frac{F^2(K_0)}{D}\int_\Omega u^2dx+c_1,
\end{split}
\end{equation*}
which yields
\begin{equation}\label{2L1-5}
\frac{d}{dt}\int_\Omega |\nabla v|^2dx+D\int_\Omega |\Delta v|^2dx\leq  \frac{2F^2(K_0)}{D}\int_\Omega u^2dx+2c_1.
\end{equation}
Using the Gagliardo-Nirenberg inequality and noting the fact $\|v\|_{L^2}\leq K_0|\Omega|^\frac{1}{2}$, one has
\begin{equation}\label{2L1-6}
\begin{split}
\int_\Omega |\nabla v|^2dx=\|\nabla v\|_{L^2}^2
\leq c_2(\|\Delta v\|_{L^2}\|v\|_{L^2}+\|v\|_{L^2}^2)
\leq \frac{D}{2} \|\Delta v\|_{L^2}^2+c_3.
\end{split}
\end{equation}
Substituting \eqref{2L1-6} into \eqref{2L1-5}, we get
\begin{equation}\label{2L1-7}
\frac{d}{dt}\int_\Omega |\nabla v|^2dx+\int_\Omega |\nabla v|^2dx+\frac{D}{2}\int_\Omega |\Delta v|^2dx\leq  \frac{2F^2(K_0)}{D}\int_\Omega u^2dx+c_4.
\end{equation}
{\color{black}{Denote} }$y(t):=\int_\Omega |\nabla v(\cdot,t)|^2dx$, $t\in [0,T_{max})$ and $z(t):=\frac{2F^2(K_0)}{D}\int_\Omega u^2dx+c_4$, from \eqref{2L1-7} we have
\begin{equation}\label{2L1-8}
y'(t)+y(t)\leq z(t) \ \mathrm{for \  all} \ \ t\in(0,T_{max}),
\end{equation}
{\color{black}which gives   \eqref{2L1-1} by noting \eqref{2L1-2}}. On the other hand, integrating  \eqref{2L1-7} over $(t,t+\tau)$ along with \eqref{2L1-2} and \eqref{2L1-1}, we obtain \eqref{2L1-3}.
\end{proof}



\section{Boundedness of solutions}
We will derive the {\it a priori} $L^2$-estimate of the solution component $u$.
\begin{lemma}\label{L2}
Let the conditions in Theorem $\ref{GB}$ hold. Then the solution of $\eqref{1-1}$ satisfies
\begin{equation}\label{2-1}
\|u(\cdot,t)\|_{L^2}\leq C, \ \text{for all} \ t \in (0,T_{\max})
\end{equation}
where $C>0$ is a constant independent of $t$.
\end{lemma}
\begin{proof} Multiplying the first equation of $\eqref{1-1}$ by $2u$,  integrating the result with respect to $x$ over $\Omega$,  one has
\begin{equation}\label{3-14}
\begin{split}
&\frac{d}{dt}\int_\Omega u^2dx+2\int_\Omega d(v) |\nabla u|^2dx+2\alpha \int_\Omega u^3dx +2\theta\int_\Omega u^2dx\\
&=2\int_\Omega u\chi (v)\nabla u\cdot\nabla vdx+2\gamma \int_\Omega u^2F(v)dx.\\
\end{split}
\end{equation}
With the  assumptions in (H1)-(H2) and the fact  \eqref{Lv}, one has
 $d(v)\geq d(K_0)>0$, $0<F(v)\leq F(K_0)$ and {\color{black}$0\leq \chi (v)\leq c_1$}. Then using the Young's inequality and H\"{o}lder inequality, we have from \eqref{3-14} that
 \begin{equation}\label{2-2}
\begin{split}
&\frac{d}{dt}\int_\Omega u^2dx +2d(K_0)\int_\Omega  |\nabla u|^2dx+2\alpha \int_\Omega u^3dx +2\theta\int_\Omega u^2dx\\
&\leq 2c_1\int_\Omega u|\nabla u||\nabla v|dx+2\gamma F(K_0) \int_\Omega u^2dx\\
&\leq  d(K_0)\int_\Omega |\nabla u|^2dx+\frac{c_1^2}{d(K_0)}\int_\Omega u^2|\nabla v|^2dx+2\gamma F(K_0) \int_\Omega u^2dx\\
&\leq  d(K_0) \|\nabla u\|_{L^2}^2+\frac{c_1^2}{d(K_0)}\|u\|_{L^4}^2\|\nabla v\|_{L^4}^2+2\gamma F(K_0) \|u\|_{L^2}^2.
\end{split}
\end{equation}
Using the Gagliardo-Nirenberg inequality, one  has
\begin{equation}\label{2-6}
\begin{split}
\|u\|_{L^4}^2
&\leq c_2\left(\|\nabla u\|_{L^2}\| u\|_{L^2}
+\| u\|_{L^2}^2\right)
\end{split}
\end{equation}
and the following estimate (cf. \cite[Lemma 2.5]{JKW-SIAP-2018})
\begin{equation}\label{2-7}
\begin{split}
\|\nabla v\|_{L^4}^2
\leq c_3\left(\|\Delta v\|_{L^2}\|\nabla v\|_{L^2}+\|\nabla v\|_{L^2}^2\right)
\leq c_3c_4(\|\Delta v\|_{L^2}+c_4)
\end{split}
\end{equation}
where the fact $\|\nabla v\|_{L^2}\leq c_4$ has been used.
Combining \eqref{2-6}-\eqref{2-7} and  using the Young's inequality, we obtain
\begin{equation}\label{2-8}
\begin{split}
\frac{c_1^2}{d(K_0)}\|u\|_{L^4}^2\|\nabla v\|_{L^4}^2
&\leq c_5\left(\|\nabla u\|_{L^2}\| u\|_{L^2}
+\| u\|_{L^2}^2\right)(\|\Delta v\|_{L^2}+c_4)\\
&\leq c_5\|\nabla u\|_{L^2}\|u\|_{L^2}\|\Delta v\|_{L^2}+c_4c_5\|\nabla u\|_{L^2}\|u\|_{L^2}\\
&\ \ \ \ +c_5\|u\|_{L^2}^2\|\Delta v\|_{L^2}+c_4c_5\|u\|_{L^2}^2\\
&\leq d(K_0) \|\nabla u\|_{L^2}^2+\frac{c_5^2}{d(K_0)}\|u\|_{L^2}^2\|\Delta v\|_{L^2}^2
+c_6\|u\|_{L^2}^2,
\end{split}
\end{equation}
where $c_5:=\frac{c_1^2c_2c_3c_4}{d(K_0)}$ and $c_6:=\frac{c_4^2c_5^2+d^2(K_0)+2c_4c_5d(K_0)}{2d(K_0)}$. Then with \eqref{2-8}, we update \eqref{2-2} as
\begin{equation}\label{2-9}
\begin{split}
\frac{d}{dt}\|u\|_{L^2}^2
\leq \frac{c_5^2}{d(K_0)}\|u\|_{L^2}^2\|\Delta v\|_{L^2}^2+(2\gamma F(K_0)+c_6)\|u\|_{L^2}^2
\leq c_7\|u\|_{L^2}^2(\|\Delta v\|_{L^2}^2+1)
\end{split}
\end{equation}
with $c_7:=\frac{c_5^2}{d(K_0)}+2\gamma F(K_0)+c_6$. For any $t\in (0,T_{max})$ and in the case of either $t\in(0,\tau)$ or $t\geq \tau$ with $\tau=\min \Big\{ 1,\frac{1}{2} T_{max}\Big\}$, from \eqref{2L1-2} we can find a $t_0=t_0(t)\in ((t-\tau)_+, t)$ such that $t_0\geq 0$ and
\begin{equation}\label{2-10}
\int_\Omega u^2(x,t_0)dx\leq c_{8}.
\end{equation}
On the other hand, from \eqref{2L1-3} in Lemma  \ref{2L1}, we can find a constant $c_{9}>0$ such that
\begin{equation}\label{2-11}
\int_{t_0}^{t_0+\tau}\int_\Omega |\Delta v(x,s)|^2dxds\leq c_{9} \ \ \mathrm{for \ all} \ \ t_0\in (0,\widetilde{T}_{max}).
\end{equation}
Then integrating \eqref{2-9} over $(t_0,t)$, and using \eqref{2-10}, \eqref{2-11} and the fact $t\leq t_0+\tau\leq t_0+1$, we  derive
\begin{equation*}\label{2L2-12}
\begin{split}
\|u(\cdot,t)\|_{L^2}^2
&\leq \|u(\cdot,t_0)\|_{L^2}^2\cdot e^{c_7\int_{t_0}^t(\|\Delta v(\cdot,s)\|_{L^2}^2+1)ds}\leq c_{8}  e^{c_7(c_9+1)},
\end{split}
\end{equation*}
which yields  \eqref{2-1} and completes the proof of  Lemma \ref{L2}.
\end{proof}

Next, we will derive the boundedness of $\|u(\cdot,t)\|_{L^\infty}$ by the Moser iteration.

\begin{lemma}\label{L7}Let the conditions in Theorem $\ref{GB}$ hold. Then the solution of system (\ref{1-1}) satisfies
\begin{equation}\label{L7-1}
\|u(\cdot,t)\|_{L^\infty}\leq C  \ \  \ \mathrm{for\  all}\  t\in(0,T_{max}),
\end{equation}
where the constant $C>0$  independent of $t$.
\end{lemma}
\begin{proof}
First, we claim that if $\|u(\cdot,t)\|_{L^p}\leq M_0 (p\geq 1)$, then it holds that
\begin{equation}\label{L3-2}
\|\nabla v(\cdot,t)\|_{L^r}\leq c_1, \ \mathrm{for\ all} \ t\in(0,T_{max})
\end{equation}
with
{\color{black} \begin{equation}\label{L3-3}
r\in\begin{cases}
[1,\frac{np}{n-p}), \ \ &\mathrm{if}\ \ \  p< n,\\
[1,\infty), \ \ &\mathrm{if}\ \ \  p= n,\\
[1,\infty],\ \ &\mathrm{if}\ \ \  p>n.\\
\end{cases}
\end{equation}
}
In fact, from the second equation of system \eqref{1-1}, we know that $v$ solves the following problem
\begin{equation}\label{L3-4}
v_t=D\Delta v-v+g(u,v) \ \ \mathrm{in}\ \  \Omega,\ \ \frac{\partial v}{\partial\nu }=0,
\end{equation}
where $g(u,v):= v-uF(v)+f(v)$. Noting the properties of $F(v), f(v)$ and the fact that {\color{black}$0\leq v(x,t)\leq K_0$} in \eqref{Lv}, one has
\begin{equation}\label{L3-5}
\|g(u,v)\|_{L^p}\leq c_2(\|u\|_{L^p}+1)\leq c_2(M_0+1):= c_3.
\end{equation}
Then applying the results of \cite[Lemma 1]{KS-JMAA} (see also \cite[Lemma 1.2]{TW-2012}) to the problem \eqref{L3-4} with  \eqref{L3-5}, we obtain  \eqref{L3-2}  with \eqref{L3-3}.

Using $u^{p-1}$ with $p\geq 2$ as a test function for  the first equation in  (\ref{1-1}), and integrating the resulting equation by parts, we obtain
\begin{equation}\label{L4-2}
\begin{split}
&\frac{1}{p}\frac{d}{dt}\int_\Omega u^pdx +(p-1)\int_\Omega d(v) u^{p-2}|\nabla u|^2dx+\alpha \int_\Omega u^{p+1}dx+\theta \int_\Omega u^pdx\\
&=(p-1)\int_\Omega \chi (v)u^{p-1}\nabla u\cdot\nabla vdx+\gamma \int_\Omega  F(v)u^pdx.\\
\end{split}
\end{equation}
Using the facts  $0\leq \chi(v)\leq c_4$, $d(v)\geq d (K_0)>0$, $0<F(v)\leq F(K_0)$,  $\alpha\geq 0$ and  applying the Young's inequality,  from \eqref{L4-2} we obtain
\begin{equation*}\label{L4-3}
\begin{split}
&\frac{1}{p}\frac{d}{dt}\int_\Omega u^pdx +(p-1)d(K_0)\int_\Omega  u^{p-2}|\nabla u|^2dx+\theta \int_\Omega u^pdx\\
&\leq c_4 (p-1) \int_\Omega u^{p-1}|\nabla u||\nabla v|dx+\gamma F(K_0) \int_\Omega u^pdx\\
&\leq \frac{(p-1)d(K_0)}{2}\int_\Omega  u^{p-2}|\nabla u|^2dx+\frac{c_4^2(p-1)}{2d(K_0)}\int_\Omega u^p|\nabla v|^2dx+\gamma F(K_0) \int_\Omega u^pdx,
\end{split}
\end{equation*}
which along with the fact  $$\frac{p(p-1)d(K_0)}{2}\int_\Omega  u^{p-2}|\nabla u|^2dx=\frac{2(p-1)d(K_0)}{p}\int_\Omega |\nabla u^\frac{p}{2}|^2dx$$ gives
\begin{equation}\label{L4-4}
\begin{split}
&\frac{d}{dt}\int_\Omega u^pdx +p\theta \int_\Omega u^{p}dx+\frac{2(p-1)d(K_0)}{p}\int_\Omega |\nabla u^\frac{p}{2}|^2dx\\
&\leq \frac{c_4^2p(p-1)}{2d(K_0)}\int_\Omega u^p|\nabla v|^2dx+\gamma F(K_0) \int_\Omega u^pdx\\
\end{split}
\end{equation}
 for all $t\in(0,T_{max})$ and for all $p\geq 2$.   From  Lemma \ref{L2}, one has  $\|u(\cdot,t)\|_{L^2}\leq c_5$ and hence $\|\nabla v(\cdot,t)\|_{L^4}\leq c_6$ by noting \eqref{L3-2}. Then using the H\"{o}lder inequality and the Gagliardo-Nirenberg inequality with the fact $\|u^\frac{p}{2}(\cdot,t)\|_{L^\frac{4}{p}}=\|u(\cdot,t)\|_{L^2}^\frac{p}{2}\leq c_5^\frac{p}{2}$,  one has for all $  t\in(0,T_{max})$
\begin{equation}\label{L7-4}
\begin{split}
\frac{c_4^2p(p-1)}{2d(K_0)}\int_\Omega u^p|\nabla v|^2dx&\leq \frac{c_4^2p(p-1)}{2d(K_0)}\left(\int_\Omega u^{2p}dx\right)^\frac{1}{2}\left(\int_\Omega |\nabla v|^4dx\right)^\frac{1}{2}\\
&\leq \frac{c_4^2c_6^2p(p-1)}{2d(K_0)}\|u^\frac{p}{2}\|_{L^4}^2 \\
&\leq c_7(\|\nabla u^\frac{p}{2}\|_{L^2}^{2(1-\frac{1}{p})}\|u^\frac{p}{2}\|_{L^\frac{4}{p}}^{\frac{2}{p}}+\|u^\frac{p}{2}\|_{L^\frac{4}{p}}^2)\\
&\leq c_7 c_5\|\nabla u^\frac{p}{2}\|_{L^2}^{2(1-\frac{1}{p})}+c_7c_5^p\\
&\leq \frac{(p-1)d(K_0)}{p}\|\nabla u^\frac{p}{2}\|_{L^2}^2+\frac{d(K_0)}{p}\left(\frac{c_7c_5}{d(K_0)}\right)^{p}
+c_7c_5^p,
\end{split}
\end{equation}
and
\begin{equation}\label{L7-5}
\begin{split}
\gamma F(K_0) \int_\Omega u^pdx
&=\gamma F(K_0)\|u^\frac{p}{2}\|_{L^2}^2\\
&\leq c_8(\|\nabla u^\frac{p}{2}\|_{L^2}^{2(1-\frac{2}{p})}\|u^\frac{p}{2}\|_{L^\frac{4}{p}}^{\frac{4}{p}}+\|u^\frac{p}{2}\|_{L^\frac{4}{p}}^2)\\
&\leq c_8 c_5^2\|\nabla u^\frac{p}{2}\|_{L^2}^{2(1-\frac{2}{p})}+c_8c_5^p\\
&\leq  \frac{(p-1)d(K_0)}{p}\|\nabla u^\frac{p}{2}\|_{L^2}^2+\frac{d(K_0)}{p}\left(\frac{c_8c_5^2}{d(K_0)}\right)^{p}
+c_8c_5^p.
\end{split}
\end{equation}
Substituting \eqref{L7-4} and \eqref{L7-5} into \eqref{L4-4}, and letting $c_9:=\frac{d(K_0)}{p}\left(\frac{c_7c_5}{d(K_0)}\right)^{p}+\frac{d(K_0)}{p}\left(\frac{c_8c_5^2}{d(K_0)}\right)^{p}+(c_7+c_8)c_5^p$, we obtain
\begin{equation*}
\frac{d}{dt}\int_\Omega u^pdx +p\theta\int_\Omega u^{p}dx\leq c_9,
\end{equation*}
which, combined with the Gronwall's inequality, yields
\begin{equation}\label{Lp}
\|u(\cdot,t)\|_{L^p}^p\leq e^{-p\theta t}\|u_0\|_{L^p}^p+\frac{c_9}{p\theta}(1-e^{-p\theta t})\leq \|u_0\|_{L^p}^p+\frac{c_9}{p\theta}.
\end{equation}
Then choosing $p=4$ in \eqref{Lp} and using \eqref{L3-2} again, one can find a constant  $c_{10}>0$ independent of $p$ such that $\|\nabla v(\cdot,t)\|_{L^\infty}\leq c_7$. Then using the Moser iteration procedure (cf. \cite{MS}), one has \eqref{L7-1}. Hence we complete the proof of Lemma \ref{L7}.
\end{proof}

\begin{proof}[Proof of Theorem $\ref{GB}$] Theorem \ref{GB} is a consequence of Lemma \ref{LS} and Lemma \ref{L7}.
\end{proof}

\section{Globally asymptotic stability of solutions}
{\color{black} Based on some ideas in \cite{JW-JDE-2017}, we shall prove the {\color{black}global stability results in} Theorem $\ref{WP}$ in this section by the method of  Lyapunov functionals with the help of LaSalle's invariant principle under the hypotheses (H1)-(H4)}. {\color{black} Here we employ the same Lyapunov functionals as in  \cite{JW-JDE-2017} and hence will skip many similar computations.}  To proceed, we first derive some regularity results for the solution $(u,v)$ by using some ideas in \cite{Tao-Winkler-SIMA-2011}.
\begin{lemma}\label{L9} Let $(u,v)$ be the nonnegative global classical solution of system \eqref{1-1} obtained in Theorem \ref{GB}.  Then there  exist $\sigma\in(0,1)$ and $C>0$ 
such that
\begin{equation}\label{Lv-1}
\|v\|_{C^{2+\sigma,1+\frac{\sigma}{2}}(\bar{\Omega}\times[t,t+1])}\leq C, \ \mathrm{for\ all}\ \  t> 1.
\end{equation}
\end{lemma}

\begin{proof}
From Theorem  \ref{GB}, we can find three positive constants $c_1,c_2,c_3$ such that
\begin{equation*}
0<u(x,t)\leq c_1, 0<v(x,t)\leq c_2 \  \ \mathrm{and}\ \ |\nabla v(x,t)|\leq c_3 \ \mathrm{for \ all }\ x\in \Omega \ \mathrm{and} \ \ t>0.
\end{equation*}
The first equation of system \eqref{1-1} can be rewritten as
\begin{equation*}\label{L9-3}
u_t=\nabla \cdot A(x,t,\nabla u)+B(x,t) \ \mathrm{for\ all}\  x\in \Omega \ \mathrm{and} \ \ t>0,
\end{equation*}
where
\begin{equation*}
A(x,t,\nabla u):=d(v) \nabla u-\chi (v)u \nabla v
\end{equation*}
and
\begin{equation*}
B(x,t):=(\gamma F(v) -\theta -\alpha u)u.
\end{equation*}
By the assumptions in (H1)-(H3) and using the Young's inequality,   then for all $x\in\Omega$ and $t>0$, we  obtain  that
\begin{equation}\label{L9-4}
\begin{split}
A(x,t,\nabla u)\cdot \nabla u
&=d(v)|\nabla u|^2-\chi(v)u\nabla v\cdot\nabla u\\
&\geq d(v)|\nabla u|^2-|\chi(v)| u|\nabla v| |\nabla u|\\
&\geq \frac{d(v)}{2}|\nabla u|^2-\frac{|\chi(v)|^2}{2d(v)} u^2|\nabla v|^2\\
&\geq \frac{d(c_2)}{2}|\nabla u|^2-c_4
\end{split}
\end{equation}
and
\begin{equation}\label{L9-5}
|A(x,t,\nabla u)|\leq d(0)|\nabla u|+c_5
\end{equation}
as well as
\begin{equation}\label{L9-6}
|B(x,t)|\leq c_6.
\end{equation}
Then \eqref{L9-4}-\eqref{L9-6} allow us to apply the H\"{o}lder regularity  for quasilinear parabolic equations \cite[Theorem 1.3 and Remark 1.4]{RT} to obtain
$
\|u\|_{{C^{\sigma,\frac{\sigma}{2}}}(\bar{\Omega}\times[t,t+1])}\leq c_7$ for all  $t> 1$. Moreover, applying the standard parabolic schauder theory \cite{La} to the second equation of  \eqref{1-1}, one has \eqref{Lv-1}. Then the proof of Lemma \ref{L9} is completed.
\end{proof}
With the results in Lemma \ref{L9} in hand, we next derive the following results.
\begin{lemma}\label{L4u}
Suppose the conditions in Lemma \ref{L9} hold. Then we can find a constant $C>0$ such that
\begin{equation}\label{L4u-1}
\|\nabla u\|_{L^4}\leq C \ \mathrm{for\  all}\ \ t>1.
\end{equation}
\end{lemma}
\begin{proof}
From the first equation of system \eqref{1-1}, we have
\begin{equation}\label{L4e-2}
\begin{split}
\frac{1}{4}\frac{d}{dt}\int_\Omega |\nabla u|^4dx
&=\int_\Omega |\nabla u|^2\nabla u\cdot \nabla u_tdx\\
&=\int_\Omega |\nabla u|^2\nabla u\cdot \nabla (\nabla \cdot(d(v)\nabla u))dx-\int_\Omega |\nabla u|^2\nabla u\cdot \nabla (\nabla \cdot(\chi(v)u\nabla v))dx\\
&\ \ \ \ + \int_\Omega \nabla (\gamma F(v) u-\theta u -\alpha u^2)\cdot\nabla u |\nabla u|^2dx \\
&=: I_1+I_2+I_3.
\end{split}
\end{equation}
With the integration by parts, the term $I_1$ becomes
\begin{equation*}\label{I1}
\begin{split}
I_1
&={\color{black} -\int_\Omega (|\nabla u|^2\Delta u)  \nabla \cdot(d(v)\nabla u)dx-\int_\Omega (\nabla |\nabla u|^2\cdot \nabla u)\nabla \cdot(d(v)\nabla u)dx}\\
&={\color{black}\int_\Omega d(v)|\nabla u|^2(\nabla \Delta u\cdot \nabla u) dx
-\int_\Omega d'(v) (\nabla |\nabla u|^2\cdot \nabla u)( \nabla u\cdot \nabla v)dx},
\end{split}
\end{equation*}
which, combined with the fact $\nabla \Delta u\cdot  \nabla  u=\frac{1}{2}\Delta |\nabla u|^2-|D^2 u|^2$ {\color{black} where $D^2 u$ denotes the Hessian matrix of $u$}, gives
\begin{equation}\label{I1-1}
\begin{split}
I_1&=\frac{1}{2}\int_\Omega d(v)|\nabla u|^2 \Delta |\nabla u|^2dx-\int_\Omega d(v) |\nabla u|^2 |D^2 u|^2dx-\int_\Omega d'(v)(\nabla |\nabla u|^2\cdot \nabla u)(\nabla u\cdot\nabla v)dx\\
&=\frac{1}{2}\int_{\partial\Omega} d(v)|\nabla u|^2\frac{\partial |\nabla u|^2}{\partial \nu}dS-\frac{1}{2}\int_\Omega d'(v)|\nabla u|^2\nabla v\cdot \nabla |\nabla u|^2dx-\frac{1}{2}\int_\Omega d(v)|\nabla |\nabla u|^2|^2dx\\
&\ \ \ -\int_\Omega d(v) |\nabla u|^2 |D^2 u|^2dx-\int_\Omega d'(v)(\nabla |\nabla u|^2\cdot \nabla u)(\nabla u\cdot\nabla v)dx\\
&\leq \frac{1}{2}\int_{\partial\Omega} d(v)|\nabla u|^2\frac{\partial |\nabla u|^2}{\partial \nu}dS-\frac{1}{2}\int_\Omega d(v)|\nabla |\nabla u|^2|^2dx -\int_\Omega d(v) |\nabla u|^2 |D^2 u|^2dx\\
&\ \ \ \ +\frac{3}{2}\int_\Omega |d'(v)|\Big|\nabla |\nabla u|^2\Big||\nabla u|^2|\nabla v|dx.
\end{split}
\end{equation}
With the facts $\|u(\cdot,t)\|_{L^\infty}+\|v(\cdot,t)\|_{W^{1,\infty}}\leq c_1$ and \eqref{Lv-1}, we have
\begin{equation*}
\begin{split}
\nabla \cdot (\chi(v)u\nabla v)
&=\chi'(v) u |\nabla v|^2+\chi(v)\nabla u\cdot \nabla v+\chi(v) u\Delta v\\
&\leq c_2(1+|\nabla u|) \ \mathrm{for\ all\ } t>1,
\end{split}
\end{equation*}
which updates $I_2$ as
\begin{equation}\label{I2}
\begin{split}
I_2
&=\int_\Omega \nabla |\nabla u|^2\cdot \nabla u \nabla \cdot (\chi(v)u\nabla v)dx+\int_\Omega |\nabla u|^2\Delta u\nabla \cdot (\chi(v)u\nabla v)dx\\
&\leq c_2\int_\Omega |\nabla u||\nabla |\nabla u|^2|(1+|\nabla u|)dx+c_2\int_\Omega |\nabla u|^2|\Delta u|(1+|\nabla u|)dx.\\
\end{split}
\end{equation}
Moreover, the term $I_3$ can be estimated as follows:
\begin{equation}\label{I3}
\begin{split}
I_3&=\int_\Omega \nabla (\gamma F(v) u-\theta u -\alpha u^2)\cdot\nabla u |\nabla u|^2dx\\
&=\int_\Omega \gamma F'(v) u\nabla v\cdot \nabla u|\nabla u|^2dx+\gamma \int_\Omega  F(v)|\nabla u|^4dx-\theta\int_\Omega |\nabla u|^4dx-2\alpha \int_\Omega u|\nabla u|^4dx\\
&\leq c_3\int_\Omega |\nabla u|^4dx+c_4.
\end{split}
\end{equation}
Substituting \eqref{I1-1}-\eqref{I3} into \eqref{L4e-2}, and using the fact $0<d(K_0)\leq d(v)\leq d(0)$,  we have
\begin{equation}\label{L4e-3}
\begin{split}
&\frac{1}{4}\frac{d}{dt}\int_\Omega |\nabla u|^4dx+\frac{d(K_0)}{2}\int_\Omega |\nabla |\nabla u|^2|^2dx +d(K_0)\int_\Omega  |\nabla u|^2 |D^2 u|^2dx\\
&\leq \frac{1}{2}\int_{\partial \Omega} d(v)|\nabla u|^2\frac{\partial |\nabla u|^2}{\partial \nu}dS+\frac{3}{2}\int_\Omega |d'(v)||\nabla |\nabla u|^2||\nabla u|^2|\nabla v|dx\\
&\ \ \ \ +c_2\int_\Omega |\nabla u||\nabla |\nabla u|^2|(1+|\nabla u|)dx+c_2\int_\Omega |\nabla u|^2|\Delta u|(1+|\nabla u|)dx\\
&\ \ \ \ +c_3\int_\Omega |\nabla  u|^4dx+c_4\\
&\triangleq J_1+J_2+J_3+J_4+c_3\int_\Omega |\nabla  u|^4dx+c_4.\\
\end{split}
\end{equation}
With the inequality $\frac{\partial|\nabla u|^2}{\partial \nu}\leq 2\lambda'|\nabla u|^2$ on $\partial\Omega$ (see \cite[Lemma 4.2]{MS-2014})
 we derive
 \begin{equation}\label{J1}
\begin{split}
J_1
 &\leq \frac{d(0)}{2}\Big|\int_{\partial\Omega}|\nabla u|^{2}\frac{\partial |\nabla u|^2}{\partial \nu}dS\Big|\\
&\leq \lambda' d(0)\||\nabla u|^2\|_{L^2(\partial \Omega)}^2\leq \frac{d(K_0)}{8}\int_\Omega |\nabla|\nabla u|^2|^2dx+c_5\int_\Omega |\nabla  u|^4dx,
\end{split}
\end{equation}
 where we have used the following trace inequality (cf. \cite[Remark 52.9]{PP-2007}):
\begin{equation*}
\|\varphi\|_{L^2(\partial\Omega)}\leq \varepsilon \|\nabla {\color{black}\varphi}\|_{L^2(\Omega)}+C_\varepsilon\|{\color{black}\varphi}\|_{L^2(\Omega)}, \ \mathrm{for\ any}\ \varepsilon >0.
\end{equation*}
By the boundedness of $\|\nabla v\|_{L^\infty}$ and Young's inequality, one derives
\begin{equation}\label{J2-J3}
\begin{split}
J_2+J_3
&\leq c_6\int_\Omega |\nabla |\nabla u|^2||\nabla u|(1+|\nabla u|)dx\\
&\leq  \frac{d(K_0)}{8}\int_\Omega |\nabla|\nabla u|^2|^2dx+c_7\int_\Omega |\nabla  u|^4dx+c_7.
\end{split}
\end{equation}
 Since $|\Delta u|\leq \sqrt{2} |D^2u|$, one can estimate $J_4$ as follows
 \begin{equation}\label{J4}
 \begin{split}
 J_4
 &\leq c_2\sqrt{2}\int_\Omega |\nabla u|^2|D^2u|dx+c_2\sqrt{2}\int_\Omega |\nabla u|^3|D^2u|dx\\
 &\leq \frac{d(K_0)}{2}\int_\Omega |\nabla u|^2|D^2 u|^2dx+\frac{2c_2^2}{d(K_0)}\int_\Omega (|\nabla u|^2+|\nabla u|^4)dx\\
 &\leq \frac{d(K_0)}{2}\int_\Omega |\nabla u|^2|D^2 u|^2dx+\frac{4c_2^2}{d(K_0)}\int_\Omega |\nabla  u|^4dx+\frac{c_2^2|\Omega|}{2d(K_0)}.
 \end{split}
 \end{equation}
Substituting \eqref{J1}-\eqref{J4} into \eqref{L4e-3}, we end up with
\begin{equation}\label{Jn}
\begin{split}
\frac{d}{dt}\int_\Omega |\nabla u|^4dx+d(K_0)\int_\Omega |\nabla |\nabla u|^2|^2dx &+2d(K_0)\int_\Omega  |\nabla u|^2 |D^2 u|^2dx\\
&\leq c_8\int_\Omega |\nabla  u|^4dx+c_9,\\
\end{split}
\end{equation}
where $c_8:= 4\left(c_3+c_5+c_7+\frac{4c_2^2}{d(K_0)}\right)$ and $c_9:=4\left(c_4+c_7+\frac{c_2^2|\Omega|}{2d(K_0)}\right)$.
On the other hand, integrating by parts, noting $\|u(\cdot,t)\|_{L^\infty}\leq c_{10}$ and  using the Young's inequality, one has
\begin{equation*}\label{L4*}
\begin{split}
(c_8+2)\int_\Omega |\nabla  u|^4dx
&=c_{11}\int_\Omega |\nabla u|^2\nabla u\cdot \nabla udx\\
&=-c_{11}\int_\Omega u \nabla |\nabla u|^2 \cdot \nabla udx-c_{11}\int_\Omega u|\nabla u|^2\Delta udx\\
&\leq c_{10}c_{11}\int_\Omega |\nabla |\nabla u|^2| |\nabla u|dx+c_{10}c_{11}\sqrt{2}\int_\Omega |\nabla u|^2|D^2 u|dx\\
&\leq d(K_0)\int_\Omega |\nabla |\nabla u|^2|^2dx+2d(K_0)\int_\Omega |\nabla u|^2|D^2 u|^2dx+\frac{c_{10}^2c_{11}^2}{2}\int_\Omega |\nabla u|^2dx\\
&\leq d(K_0)\int_\Omega |\nabla |\nabla u|^2|^2dx+2d(K_0)\int_\Omega |\nabla u|^2|D^2 u|^2dx\\
&\ \ \ \ + \int_\Omega |\nabla u|^4dx+\frac{|\Omega |c_{10}^4c_{11}^4}{16},\\
\end{split}
\end{equation*}
which updates \eqref{Jn} to
\begin{equation}\label{L4e-4}
\frac{d}{dt}\int_\Omega |\nabla u|^4dx+\int_\Omega |\nabla u|^4dx\leq c_9+\frac{c_{10}^4c_{11}^4|\Omega |}{16}.
\end{equation}
Applying the Gronwall's inequality to \eqref{L4e-4} along with the fact $\|\nabla u(\cdot,1)\|_{L^\infty}\leq c_{12}$, one obtains \eqref{L4u-1}. Thus the proof of Lemma \ref{L4u} is finished.
\end{proof}
Next, we state a basic result which will be used later.
\begin{lemma}[\cite{JW-JDE-2017}]\label{blemma}
Let $F$ satisfy the {\color{black}conditions} in (H2) and define a function for some constant {\color{black}$\omega_*> 0$}:
$$\zeta(v)=\int_{\omega_*}^v \frac{F(\sigma)-F(\omega_*)}{F(\sigma)}d\sigma.$$
Then $\zeta(v)$ is a convex function such that $\zeta(v)\geq 0$. If $(u,v)$ is a solution of (\ref{1-1}) satisfying $v \to \omega_*$ as $t \to \infty$, then there is a constant $T_0>0$ such that for all $t \geq T_0$ it holds that
\begin{equation*}\label{C1-9n}
\frac{F'(\omega_*)}{4F(\omega_*)}(v-\omega_*)^2\leq\zeta(v)=\int_{\omega_*}^v \frac{F(\sigma)-F(\omega_*)}{F(\sigma)}d\sigma\leq \frac{F'(\omega_*)}{F(\omega_*)}(v-\omega_*)^2.
\end{equation*}
\end{lemma}

\subsection{Global stability of the prey-only steady state}
In this subsection, we shall prove that $(u,v)$ converges to $(0,K)$ in $L^\infty$ as $t \to \infty$ when $\gamma F(K)\leq\theta$
and further show that the convergence rate is exponential if $\gamma F(K)<\theta$ and  algebraic if $\gamma F(K)=\theta$ and  $\alpha>0$.
\begin{lemma}\label{wp1}
Let the assumptions in Theorem \ref{WP} hold.  If  $\gamma F(K)\leq \theta$, then the solution $(u,v)$ of $\eqref{1-1}$ satisfies $$\|u\|_{L^\infty}+\|v-K\|_{L^\infty} \to 0 \ \text{as} \ t \to \infty,$$
where the convergence rate is exponential if $\gamma F(K)< \theta$ and algegraic if $\gamma F(K)=\theta$ and $\alpha>0$.
   \end{lemma}
\begin{proof}
{\color{black}
We start with the following functional
\begin{equation}\label{lf1}
{V}_1(u(t),v(t))=: V_1(t)={\color{black}\frac{1}{\gamma}\int_\Omega udx+\int_\Omega \int_K^{v} \frac{F(\sigma)-F(K)}{F(\sigma)}d\sigma dx}
\end{equation}
and show  $\frac{d }{dt}V_1(u,v)=\frac{d}{dt}V_1(t)\leq 0$  as well as $\frac{d }{dt}V_1(u,v)=0$ iff $(u,v)=(0,K)$.}
Indeed differentiating the functional (\ref{lf1}) with respect to $t$ and using the equations in $\eqref{1-1}$, one has
\begin{equation}\label{S1}
\begin{split}
{\color{black}\frac{d}{dt}V_1(t)}
&=\frac{1}{\gamma}\int_\Omega u_tdx+\int_\Omega \frac{F(v)-F(K)}{F(v)}v_tdx\\
&=\frac{1}{\gamma}\int_\Omega u(\gamma F(v)-\theta-\alpha u)dx+\int_\Omega \bigg(1-\frac{F(K)}{F(v)}\bigg)(D \Delta v-uF(v)+f(v))dx.
\end{split}
\end{equation}
For the second term on the right hand side of (\ref{S1}), we use the integration by parts with some calculations and cancelations to have
\begin{equation*}\label{S2}
\begin{split}
\frac{d}{dt}V_1(t)
&=\frac{1}{\gamma}\int_\Omega  u(\gamma F(K)-\theta-\alpha u)dx-D F(K) \int_\Omega F'(v)\bigg|\frac{\nabla v}{F(v)}\bigg|^2dx\\
& \ \ +\int_{\Omega} \frac{f(v)}{F(v)}(F(v)-F(K))dx.
\end{split}
\end{equation*}
The rest of the proof only depends on the assumption (H2)-(H4) and hence we can follow the exact procedures in the proof of \cite[Lemma 4.2]{JW-JDE-2017} along with the LaSalle's invariant principle (cf. \cite{Kelley-Peterson, LaSalle}) (or {\color{black} the  compact method together with the Lyapunov functional as in \cite{W-APL-2018,WW-ZAMP-2018}}) to prove that $(0,K)$ is globally asymptotically stable if $\gamma F(K)\leq \theta$ and the following convergence
 \begin{equation}\label{DR-1}
 \begin{cases}
 \|u\|_{L^1}+\|v-K\|_{L^2}^2\leq c_1 e^{-c_1 t}, &\ \mathrm{if}\  \gamma F(K)< \theta,\\
 \|u\|_{L^1}+\|v-K\|_{L^2}^2\leq c_2 (1+t)^{-1},&\ \mathrm{if}\  \gamma F(K)= \theta,\alpha>0
 \end{cases}
 \end{equation}
hold for $t>t_0$ with some $t_0>1$. Furthermore  the Gagliardo-Nirenberg inequality and the result $\|\nabla u\|_{L^4}\leq c_3$  for $t>t_0$ (cf. Lemma \ref{L4u} ) entail that
 \begin{equation}\label{ar1}
 \begin{split}
 \|u\|_{L^\infty}
 \leq c_4\|\nabla u\|_{L^4}^\frac{4}{5}\|u\|_{L^1}^\frac{1}{5}+c_4\|u\|_{L^1}
 \leq c_4c_3^\frac{4}{5}\|u\|_{L^1}^\frac{1}{5}+c_4\|u\|_{L^1},
 \end{split}
 \end{equation}
 which together with \eqref{DR-1} yields the decay rate of $u$. Similarly with  $\|\nabla v\|_{L^4}\leq c_5$ (cf. \eqref{L3-2}), we obtain
 \begin{equation}\label{ar2}
 \begin{split}
 \|v-K\|_{L^\infty}
 &\leq c_6 \|\nabla v\|_{L^4}^\frac{2}{3}\|v-K\|_{L^2}^\frac{1}{3}+c_6\|v-K\|_{L^2}\\
 &\leq c_6c_5^\frac{2}{3}\|v-K\|_{L^2}^\frac{1}{3}+c_6\|v-K\|_{L^2},
 \end{split}
 \end{equation}
 which combined with \eqref{DR-1} gives the decay rate of $v$. This completes the proof of Lemma \ref{wp1}.
\end{proof}

\subsection{Global stability of the co-existence steady state}
Now we turn to the case $\gamma F(K)>\theta$ and prove the {\color{black} homogeneous} coexistence steady state $(u_*, v_*)$ is globally asymptotically stable under certain conditions. We shall prove our result based on the following Lyapunov functional as in \cite{JW-JDE-2017}:
\begin{equation*}\label{lf2t}
{\color{black}{V}_2(u(t),v(t))=: V_2(t)=\frac{1}{\gamma}\int_\Omega \bigg(u-u_*-u_*\ln \frac{u}{u_*}\bigg)dx+\int_\Omega \int_{v_*}^{v} \frac{F(\sigma)-F(v_*)}{F(\sigma)}d\sigma dx.}
\end{equation*}
\begin{lemma}\label{wp2}
Assume the assumptions in Theorem \ref{WP} hold.  If  $\gamma F(K)>\theta$ and
\begin{equation*}
D \geq\max\limits_{0\leq v\leq K_0}{\color{black} \frac{ u_* |F(v)|^2|\chi(v)|^2}{4 \gamma F(v_*)F'(v)d(v)}}
\end{equation*}
where $u_*$ and $v_*$ are determined by (\ref{ss2}) and independent of  $D$, then the solution $(u,v)$ of $\eqref{1-1}$ satisfies
 \begin{equation*}\label{sdrn}
 \|u-u_*\|_{L^\infty}+\|v-v_*\|_{L^\infty} \to 0 \ \text{as} \ t \to \infty
 \end{equation*}
where the convergence is exponential when $\alpha>0$.
\end{lemma}

\begin{proof}
First by the same argument as the proof of  \cite[Lemma 4.3]{JW-JDE-2017}, we have that $V_2(t)\geq 0$ for all $u,v \geq 0$.
Next we differentiate $V_2(t)$ with respect to $t$ and use the equations of (\ref{1-1}) to obtain that
\begin{equation*}\label{S3}
\begin{split}
\frac{d}{dt}V_2(t)
&=\frac{1}{\gamma}\int_\Omega \bigg(1-\frac{u_*}{u}\bigg)u_tdx+\int_\Omega \frac{F(v)-F(v_*)}{F(v)}v_tdx\\
&=\underbrace{-\frac{u_*}{\gamma}\int_\Omega \frac{d(v)|\nabla u|^2}{u^2}dx-{\color{black}D F(v_*) \int_\Omega F'(v)\bigg|\frac{\nabla v}{F(v)}\bigg|^2dx}+\frac{ u_*}{\gamma}\int_\Omega \frac{\chi (v)\nabla u \cdot\nabla v}{u}dx}_{I_1}\\
& \ \ \ \ +\underbrace{{\color{black}\frac{1}{\gamma}\int_\Omega \bigg(1-\frac{u_*}{u}\bigg)(\gamma uF(v)-\theta u-\alpha u^2)dx+\int_{\Omega} (F(v)-F(v_*))\bigg(\frac{f(v)}{F(v)}-u\bigg)}dx}_{I_2}.
\end{split}
\end{equation*}
$I_1$ can be rewritten as
$${I_1=-\int_\Omega\Theta^T A \Theta},\ \Theta=\begin{bmatrix}\nabla u\\[1mm]\nabla v \end{bmatrix}, \ \ A=\begin{bmatrix}\frac{u_* d(v)}{\gamma u^2} & -\frac{\chi(v) u_*}{2 \gamma u} \\[1mm]-\frac{\chi(v) u_*}{2\gamma  u} &\frac{D F(v_*)F'(v)}{|F(v)|^2} \end{bmatrix}.$$
where $\Theta^T$ denotes the transpose of $\Theta$. Then it can be easily checked with Sylvester's criterion that the matrix $A$ is non-negative definite (and hence $I_1\leq 0$) if and only if
\begin{equation*}\label{con}
\frac{D F(v_*)F'(v)u_* d(v)}{\gamma |F(v)|^2 u^2} \geq \frac{|\chi(v)|^2 u_*^2}{4 \gamma^2 u^2}\ \ \mathrm{or}\ \ D \geq{\color{black} \frac{ u_* |F(v)|^2|\chi(v)|^2}{4 \gamma F(v_*)F'(v)d(v)}},
\end{equation*}
where $u_*$ and $v_*$  do not depend on $D$, see (\ref{ss2}).  Note that the above $I_2$ is exactly the same as the $I_2$ in the proof of \cite[Lemma 4.3]{JW-JDE-2017}. Hence we can follow the same procedure for the proof of \cite[Lemma 4.3]{JW-JDE-2017} to show the {\color{black} homogeneous} coexistence state $(u_*,v_*)$ is globally asymptotically stable and satisfies
\begin{equation}\label{L2d}
\|u-u_*\|_{L^2}+\|v- v_*\|_{L^2}\leq c_1e^{-c_2 t}\  \mathrm{for}\  t>t_0
\end{equation}
with some $t_0>1$ for $\alpha>0$. Using the Gagliardo-Nirenberg inequality and the boundedness of $\|\nabla u\|_{L^4}$ and $\|\nabla v\|_{L^\infty}$ (see Lemma \ref{L9} and Lemma \ref{L4u}), we use a similar procedure as for \eqref{ar1}-\eqref{ar2} to finally obtain the exponential decay rate in $L^\infty$-norm from \eqref{L2d}.
\end{proof}

\subsection{Proof of Theorem \ref{WP}}
Theorem \ref{WP} is directly obtained from Lemma \ref{wp1} and Lemma \ref{wp2}.

\section{Applications and spatio-temporal patterns}
The first purpose of this section is to apply our general results obtained in Theorem \ref{GB} and Theorem \ref{WP} to two most widely used predator-prey interactions: Lotka-Volterra type (i.e. $F(v)=v$) and Rosenzweig-MacArthur type (i.e. $F(v)=\frac{v}{\lambda +v}$ and $\alpha=0$) \cite{RM-1963}. Note that these results only give the global existence of solutions (by Theorem \ref{GB}) and global stability of constant steady states (by Theorem \ref{WP}), the distribution of the predator and the prey in space and the time-asymptotic dynamics of the population outside the parameter regimes found in Theorem \ref{WP} are unclear, but indeed they are more interesting from the application point of view in ecology though hard to study analytically. Hence the second purpose of this section is to numerically exploit the spatio-temporal patterns generated by the Lotka-Volterra and Rosenzweig-MacArthur predator-prey systems, which not only display the distribution patterns of the predator and the prey in space or evolution of population in time, but also provide useful sources for the future research.
\subsection{Application of our results}
In this subsection, we shall give some examples to illustrate  the applications of Theorem \ref{GB} and Theorem \ref{WP}.  The first example is the  Lotka-Volterra predator-prey system \cite{Lotka} with prey-taxis
\begin{equation}\label{LV}
\begin{cases}
u_t=\nabla\cdot(d(v) \nabla u)-\nabla \cdot(u\chi(v)\nabla v)+\gamma u v-\theta u-\alpha u^2, &x\in \Omega, ~~t>0,\\
 v_t=D\Delta v- uv+ \mu v(1-\frac{v}{K}),& x\in \Omega, ~~t>0,\\
 \frac{\partial u}{\partial\nu}=\frac{\partial v}{\partial\nu}=0,& x\in \partial\Omega, ~~t>0,\\
  u(x,0)=u_0(x),v(x,0)=v_0(x),&x\in\Omega.\\
 \end{cases}
\end{equation}
Then the application of the results in Theorem \ref{GB} and Theorem \ref{WP} yield the following results on the  system (\ref{LV}).
\begin{proposition}
Let $\Omega\subset \R^2$ be a bounded domain with smooth boundary and assume $(u_0,v_0)\in [W^{1,p}(\Omega)]^2$ with
$u_0, v_0 \geq 0(\not \equiv 0)$ and $p>2$. {\color{black}Assume $\alpha>0$,} then the problem (\ref{LV}) has a unique global classical solution $(u,v) \in [C([0,\infty)\times\bar{\Omega}) \cap C^{2,1}((0,\infty)\times\bar{\Omega})]^2$ such that:
\begin{itemize}
\item {\color{black}If $\gamma K\leq \theta$}, then
\begin{equation*}
\|u\|_{L^\infty}+\|v-K\|_{L^\infty} \to 0 \ \text{as} \ t \to \infty
\end{equation*}
where the convergence is exponential if $\gamma K<\theta$ and algebraic if $\gamma K=\theta$.
\item If $\gamma K>\theta$ and
\begin{equation*}\label{LV-1}
\frac{4D\gamma K(\mu\alpha+\theta)}{\mu (\gamma K-\theta)K_0^2} \geq M_0,
\end{equation*}
with $M_0=\max\limits_{0\leq v\leq K_0}\frac{|\chi(v)|^2}{d(v)}$, then
\begin{equation*}
\|u-u^*\|_{L^\infty}+\|v-v^*\|_{L^\infty} \to 0 \ \text{exponentially} \ \text{as} \ t \to \infty
\end{equation*}
where
\begin{equation}\label{coess}
(u_*,v_*)=\left(\frac{\mu(\gamma K-\theta)}{\gamma K+\mu\alpha}, \frac{K(\mu\alpha+\theta)}{\gamma K+\mu\alpha}\right ).
\end{equation}
\end{itemize}
\end{proposition}
The second example is the Rosenzweig-MacArthur type predator-prey interaction. From Theorem \ref{GB} with $\alpha=0$, we only have the results for the special case $\chi(v)=-d'(v)$ (density suppressed motility) which leads to the following system
\begin{equation}\label{LV-2}
\begin{cases}
u_t=\Delta(d(v)u)+ \frac{\gamma u v}{\lambda+v}-\theta u, &x\in \Omega, ~~t>0,\\
 v_t=D\Delta v- \frac{uv}{\lambda+v}+ \mu v(1-\frac{v}{K}),& x\in \Omega, ~~t>0,\\
 \frac{\partial u}{\partial\nu}=\frac{\partial v}{\partial\nu}=0,& x\in \partial\Omega, ~~t>0,\\
  u(x,0)=u_0(x),v(x,0)=v_0(x),&x\in\Omega.\\
 \end{cases}
\end{equation}
In this case, the hypothesis (H4) is satisfied by requiring $\lambda>K$.
Then the interpretation of our results of Theorem \ref{GB} and Theorem \ref{WP} into \eqref{LV-2} yields the following results.
\begin{proposition}
Assume $(u_0,v_0)\in [W^{1,p}(\Omega)]^2$ with
$u_0, v_0 \geq 0(\not \equiv 0)$ and $p>2$. If $\lambda>K$,  then the problem (\ref{LV-2}) has a unique global classical solution in $\Omega\subset \R^2$ such that
\begin{itemize}
\item If $ \frac{\gamma K}{\lambda+K}< \theta$, then
\begin{equation*}
\|u\|_{L^\infty}+\|v-K\|_{L^\infty} \to 0 \ \text{exponentially} \ \text{as} \ t \to \infty.
\end{equation*}

\item If $\frac{\gamma K}{\lambda+K}> \theta$ and
$
\frac{4D\theta\lambda}{K_0^2 u_*} \geq M_1
$
with $M_1=\max\limits_{0\leq v\leq K_0}\frac{|d'(v)|^2}{d(v)}$, it follows that
\begin{equation*}
\|u-u^*\|_{L^\infty}+\|v-v^*\|_{L^\infty} \to 0 \ \text{as} \ t \to \infty
\end{equation*}
where
$$
(u_*,v_*)=\left(\frac{\gamma \lambda \mu [(\gamma-\theta )K-\theta\lambda]}{(\gamma-\theta )^2K},\frac{\theta \lambda}{\gamma-\theta}\right).
$$
\end{itemize}
\end{proposition}

\subsection{Linear instability analysis}
In this section, we will study the possible pattern formation generated by the system \eqref{1-1} {\color{black} under the assumptions (H1)-(H3).} We begin with the space-absent ODE system of (\ref{1-1}):
\begin{equation}\label{ODE}
\begin{cases}
u_t=\gamma uF(v)-\theta u-\alpha u^2, \\
 v_t=- u F(v)+ f(v),\\
  \end{cases}
\end{equation}
which has three possible equilibria: $(0,0), (0,K)$ and $(u_*, v_*)$.
One can easily check that the steady state $(0,0)$ is linearly unstable and $(0,K)$ is linearly stable for both Lotka-Volterra and Rosenzweig-MacArthur type predator-prey interactions. For the {\color{black} homogeneous} coexistence steady state $(u_*, v_*)$, it is linearly stable for the Lotka-Volterra type interaction. While for the Rosenzweig-MacArthur type interaction, one can easily find that all the eigenvalues of the linearised system of (\ref{ODE}) at $(u_*, v_*)$ have negative real part  (hence  $(u_*, v_*)$ is stable)  if $v_*>\frac{K-\lambda}{2}$, and are complex with zero real part if $v_*=\frac{K-\lambda}{2}$, and  are complex with  positive real part (hence $(u_*, v_*)$ is unstable) if $0<v_*<\frac{K-\lambda}{2}$.

Next we proceed to consider the stability of equilibria $(0,K)$ and  $(u_*, v_*)$ in the presence of spatial structure. To this end,
we first linearize the  system \eqref{1-1} at an equilibrium  $(u_s,v_s)$ and write the linearized system as
\begin{equation}\label{PF-1}
\begin{cases}
\Phi_t=\mathcal{A}\Delta \Phi+\mathcal{B}\Phi, &x\in\Omega, \ t>0,\\
(\nu\cdot \nabla)\Phi=0,&x\in\partial\Omega, \ t>0,\\
\Phi(x,0)=(u_0-u_s,v_0-v_s)^{\mathcal{T}},&x\in \Omega,
\end{cases}
\end{equation}
where $\mathcal{T}$ denotes the transpose and
       \begin{equation*}
        \Phi=\left(
         \begin{array}{c}
           u-u_s\\
           v-v_s\\
         \end{array}
       \right),
       \ \ \ \  \mathcal{A}=\left(
       \begin{array}{cc}
     d(v_s) & -u_s\chi(v_s)\ \\
      0 & D \\ \end{array}
       \right)
     \end{equation*}
     as well as
     \begin{equation*}
     \mathcal{B}=\left(
       \begin{array}{cc}
       -\alpha u_s & \gamma u_s F'(v_s)\\
         -F(v_s) &   -u_s F'(v_s)+f'(v_s) \\
       \end{array}
       \right)=\left(
       \begin{array}{cc}
       B_1 & B_2\\
         B_3 &  B_4 \\
       \end{array}
       \right).
     \end{equation*}
 Let $W_k(x)$ denote the eigenfunction of the following eigenvalue problem:
 \begin{equation*}\label{hn}
 \Delta W_k(x)+k^2 W_k(x)=0,  \ \  \frac{\partial W_k(x)}{\partial \nu}=0,
 \end{equation*}
 where  $k$ is called the wavenumber. Since the system \eqref{PF-1} is linear, the solution $\Phi(x,t)$ has the form of
 \begin{equation}\label{PF-2}
 \Phi(x,t)=\sum\limits_{k\geq 0} c_k e^{\rho t} W_k(x)
 \end{equation}
 where the constants $c_k$ are determined by the Fourier expansion of the initial conditions in terms of $W_k(x)$ and $\rho$ is the temporal eigenvalue.
 Substituting \eqref{PF-2} into \eqref{PF-1}, one has
 \begin{equation*}
 \rho W_k(x)=-k^2 AW_k(x)+BW_k(x),
 \end{equation*}
 which implies $\rho$ is the eigenvalue of the following matrix
 \begin{equation*}
 \begin{split}
 M_k
 &=\left(
       \begin{array}{cc}
      -d(v_s)k^2-\alpha u_s & u_s\chi (v_s)k^2+\gamma u_s F'(v_s)\ \\
   -F(v_s) & -Dk^2 -u_s F'(v_s)+f'(v_s)\\ \end{array}
       \right)\\
       &=\left(
       \begin{array}{cc}
      -d(v_s)k^2+B_1 & u_s\chi (v_s)k^2+B_2\ \\
   B_3 & -Dk^2+B_4\\ \end{array}
       \right).
       \end{split}
 \end{equation*}
Calculating the eigenvalue of matrix $M_k$, we get the eigenvalues $\rho(k^2)$ as functions of the wavenumber $k$ as the roots of
  \begin{equation*}\label{LP}
      \rho^2+a(D,k^2)\rho+b(D,k^2)=0,
     \end{equation*}
     where
     \begin{equation*}\label{LP-1}
     \begin{split}
     a(D,k^2)&=(d(v_s)+D) k^2+(\alpha+F'(v_s)) u_s-f'(v_s)\\
               &=(d(v_s)+D) k^2-(B_1+B_4)
               \end{split}
     \end{equation*}
     and
     \begin{equation*}
     \ \ b(D,k^2)=d(v_s)Dk^4-(d(v_s)B_4+u_s\chi(v_s)B_3+B_1D)k^2+B_1B_4-B_2B_3.
        \end{equation*}
Then it can be easily verified that the eigenvalue $\rho$ for the prey-only steady state $(0,K)$ has negative real part for both Lotka-Volterra and Rosenzweig-MacArthur type predator-prey interaction and hence $(0,K)$ is linearly stable. Thus the pattern (if any) can only arise from the {\color{black} homogeneous} coexistence steady state $(u_*, v_*)$. In the following results,  we first show that the Lotka-Volterra type predator-prey system \eqref{LV} indeed has no pattern bifurcated from $(u_*, v_*)$.
\begin{lemma}
The {\color{black} homogeneous} coexistence steady state $(u_*,v_*)$ of system \eqref{LV} is linearly stable if
$\gamma K>\theta$.
\end{lemma}
\begin{proof} Since $F(v)=v$ and $f(v)=\mu  v(1-v/K)$, we can easily check that
\begin{equation*}
B_1=-\alpha u_*<0, B_2=\gamma u_*>0, B_3=-v_*<0, B_4=-\frac{\mu}{K} v_*<0.
\end{equation*}
Hence $B_1+B_4<0$ and $B_1B_4-B_2B_3>0$ as well as
\begin{equation*}
d(v_*)B_4+u_*\chi(v_*)B_3+B_1D<0,
\end{equation*}
which imply $a(D,k^2)>0$ and $b(D,k^2)>0$. Then the corresponding characteristic equation only has eigenvalue with negative real part. Hence $(u_*,v_*)$ is linearly stable for all  $\gamma K>\theta$.
\end{proof}
Therefore we are left to consider the possibility of the patterns bifurcated from $(u_*, v_*)$ for the Rosenzweig-MacArthur type predator-prey  system \eqref{LV-2} {\color{black}{ and hence hereafter $\alpha=0$}}. In this case, the corresponding characteristic equation is
\begin{equation}\label{LP}
      \rho^2+a(D,k^2)\rho+b(D,k^2)=0,
     \end{equation}
with
\begin{equation}\label{coee}
a(D,k^2)=(d(v_*)+D)  k^2-\beta_1,\  b(D,k^2)=d(v_*)Dk^4-
   \beta_2 k^2   +\beta_3.
\end{equation}
where
\begin{eqnarray}\label{beta}
\begin{cases}
\beta_1=\D\frac{\mu v_*(K-\lambda-2v_*)}{K(\lambda+v_*)},\\[2mm]
\beta_2=\D\frac{\mu v_*(K-\lambda-2v_*)d(v_*)}{K(\lambda+v_*)}
-{\color{black}\frac{\mu v_*(K-v_*)\chi(v_*)}{K}},\\[2mm]
\beta_3=\D\frac{\lambda\theta\mu(K-v_*)}{K(\lambda+v_*)}.
\end{cases}
\end{eqnarray}
If $K>v_*> \frac{K-\lambda}{2}$, one has $K-\lambda-2v_*<0$ which yields $a(D,k^2)>0$ and $b(D,k^2)>0$ for all  $k\geq 0$.
Hence all the eigenvalues have negative real parts and the  {\color{black} homogeneous} coexistence steady state $(u_*,v_*)$ is linearly stable.
This implies that the pattern formation is possible only when
\begin{equation}\label{star}
0<v_*\leq \frac{K-\lambda}{2}.
\end{equation}
{\color{black}First, if $v_*=\frac{K-\lambda}{2}$,  one has $a(D,0)=0$ and $b(D,0)=\beta_3>0$, which corresponds to the
spatially homogeneous periodic solutions. Moreover, for any $k>0$, one has $a(D,k)>0$ and $b(D,k)>0$, which indicates that all the eigenvalues of \eqref{LP} have negative real part and hence no spatially inhomogeneous patterns arise in this case. Next, we shall exploit whether the spatially inhomogeneous  patterns may arise in the case $0<v_*< \frac{K-\lambda}{2}$.
}

Define the set
\begin{equation*}\label{CS}
\mathcal{H}=\{(D,\eta)\in \R^2_+:a(D,\eta)=0  \}
\end{equation*}
as the Hopf bifurcation curve, and
\begin{equation*}\label{CH}
\mathcal{S}=\{(D,\eta)\in \R^2_+:b(D,\eta)=0  \}
\end{equation*}
as the steady state bifurcation curve, where the linearized system around the {\color{black} homogeneous} coexistence
steady state has an eigenvalue with zero real part on the curves $\mathcal {S}$ or $\mathcal{H}$. Furthermore if we define
\begin{equation}\label{DH}
D_{\mathcal{H}}(\eta)=\frac{\mu v_*(K-\lambda-2v_*)}{K(\lambda+v_*)\eta}-d(v_*)
\end{equation}
and
\begin{equation}\label{DS}
D_{\mathcal{S}}(\eta)=\left(\frac{\mu v_*(K-\lambda-2v_*)}{K(\lambda+v_*)}
-{\color{black}\frac{\mu v_*(K-v_*)\chi(v_*)}{d(v_*)K}}\right)\frac{1}{\eta}   -\frac{\lambda\theta\mu(K-v_*)}{Kd(v_*)(\lambda+v_*)}\frac{1}{\eta^2},
\end{equation}
then we have the following stability results for the coexistence steady state $(u_*.v_*)$.
\begin{lemma}
 Assume the parameters $\theta,\gamma,\mu,\lambda$ and $K>\lambda$ are fixed.
 Let $D_{\mathcal{H}}(\eta)$ and $D_{\mathcal{S}}(\eta)$ be defined in \eqref{DH} and \eqref{DS}. Then the {\color{black} homogeneous} coexistence steady state is locally asymptotically stable provided the parameter $D$ satisfies
 \begin{equation}\label{ss-1}
 D>\max_{\eta\geq 0}\{D_{\mathcal{H}}(\eta),D_{\mathcal{S}}(\eta)\}.
 \end{equation}
\end{lemma}
\begin{remark}\em{
In one dimension, say $\Omega=[0,\ell]$,  $\eta=k^2=(\frac{n\pi}{\ell})^2$ where $n=0,1,2,\cdots$,  and $\ell$ is the length of the domain. For fixed $D>0$, the condition \eqref{ss-1} can be satisfied if $\ell$ is small enough.  Hence, no pattern formation will develop  if $\ell$ is sufficient small.}
\end{remark}

Next we explore the possible (local) bifurcation by treating $D$ as a bifurcation parameter. Under (\ref{star}), we see that the sign of $a(D, k^2)$ and $b(D, k^2)$ could be generic and different type bifurcation may arise. In particular the discriminant of \eqref{LP} $\Delta=|a(D,k^2)|^2-4b(D,k^2)$ has not determined sign. Therefore there are two possibilities: $\Delta<0$ and $\Delta\geq 0$, where the former will lead to a Hopf bifurcation (periodic patterns) and the latter may lead to a steady state bifurcation (aggregation patterns). Note the allowable wave numbers $k$ are discrete in a bounded domain, for instance if $\Omega=(0,\ell)$ then $k=\frac{n\pi}{\ell}$ ($n=0,1,2,\cdots$.  It can be easily verified that
$$\Delta=|a(D,k^2)|^2-4b(D,k^2)=(D-d(v_*))^2k^4-2[(D+d(v_*))\beta_1-2\beta_2]k^2+\beta_1^2-4\beta_3.$$
Hence the Hopf bifurcation may occur (i.e. $\Delta<0$) if $[(D+d(v_*))\beta_1-2\beta_2]^2-(\beta_1^2-4\beta_3)(D-d(v_*))^2>0$ and there is allowable wave number $k$ such that
$$k_1^-<k^2<k_1^+$$
where $k_1^{\pm}=\frac{(D+d(v_*))\beta_1-2\beta_2\pm 2\sqrt{[(D+d(v_*))\beta_1-2\beta_2]^2-(\beta_1^2-4\beta_3)(D-d(v_*))^2}}{(D-d(v_*))^2}$ if $D\ne d(v_*)$.

If $\Delta>0$, steady state bifurcation may occur but the possibility can be generic. Indeed one can readily verify that the steady state bifurcation will occur if there is allowable wave number $k$ such that one of the following cases holds:
\begin{eqnarray}\label{SSB}
\begin{aligned}
&(i)\ b(D, k^2)<0;\\
&(ii)\ b(D, k^2)=0, {\color{black}a(D, k^2)<0};\\
&(iii)\ b(D, k^2)>0, a(D, k^2)<0 \ \text{and}\ |a(D,k^2)|^2-4b(D,k^2)>0;
\end{aligned}
\end{eqnarray}
The corresponding parameter regime guaranteeing each of {\color{black}$(i), (ii)$ and  $(iii)$} can be found with easy calculations. For example, under (\ref{star}), it follows that $\beta_3>0$. Hence condition $(i)$ is ensured if
\begin{equation}\label{ss-3}
\beta_2>0, \ \ \text{namely} \ \ {\color{black}\frac{K-\lambda-2v_*}{(\lambda+v_*)(K-v_*)}> \frac{\chi(v_*)}{d(v_*)}}
\end{equation}
and the allowable wave number $k$ satisfy
\begin{equation}\label{ss-4}
k_2^-<k^2<k_2^+,\ k_2^{\pm}=\frac{\beta_2\pm\sqrt{\Lambda}}{2Dd(v_*)} \ \ \text{with} \ \Lambda=\beta_2^2-4\beta_3 D d(v_*).
\end{equation}
Conditions of ensuring $(ii)$ or $(iii)$ can be derived similarly and will not be detailed here since  these conditions can be easily inspected when the parameter values and the motility function $d(v), \chi(v)$ are specified, as shown in the next subsection.

\begin{figure}[t]
\centering
\includegraphics[width=6cm]{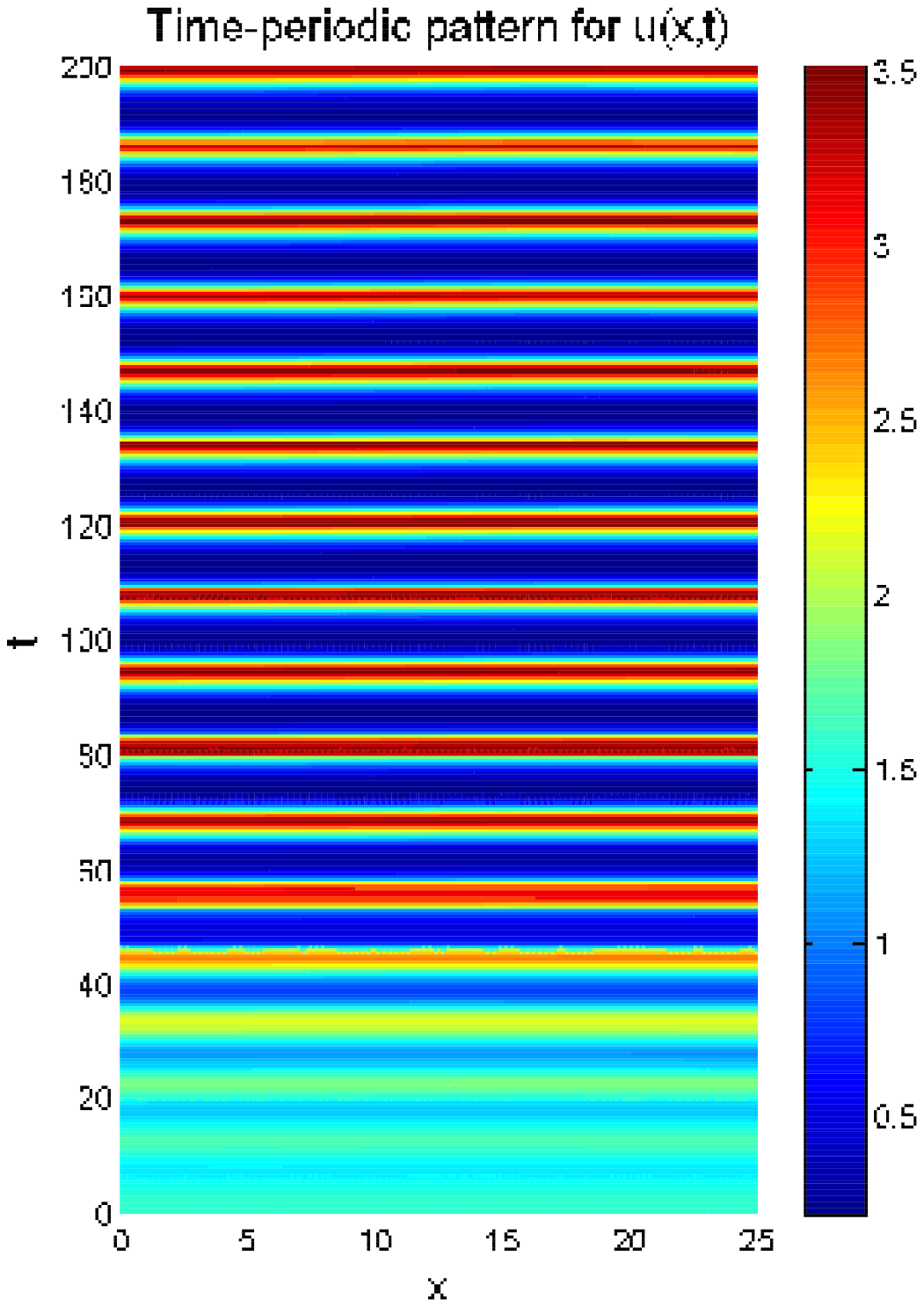}
\includegraphics[width=6cm]{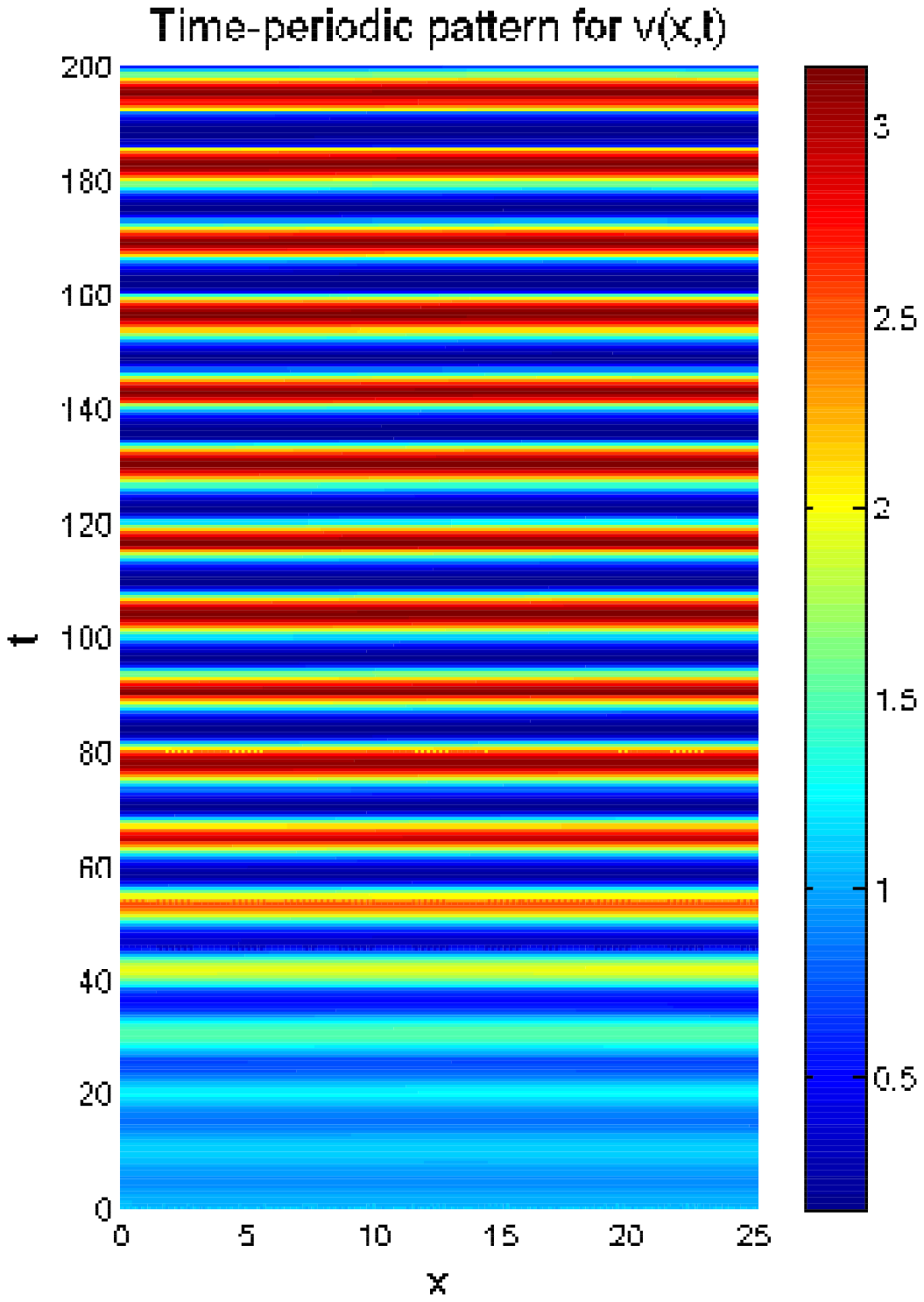}

(a) \hspace{5.5cm} (b)

\vspace{0.4cm}

\includegraphics[width=12cm]{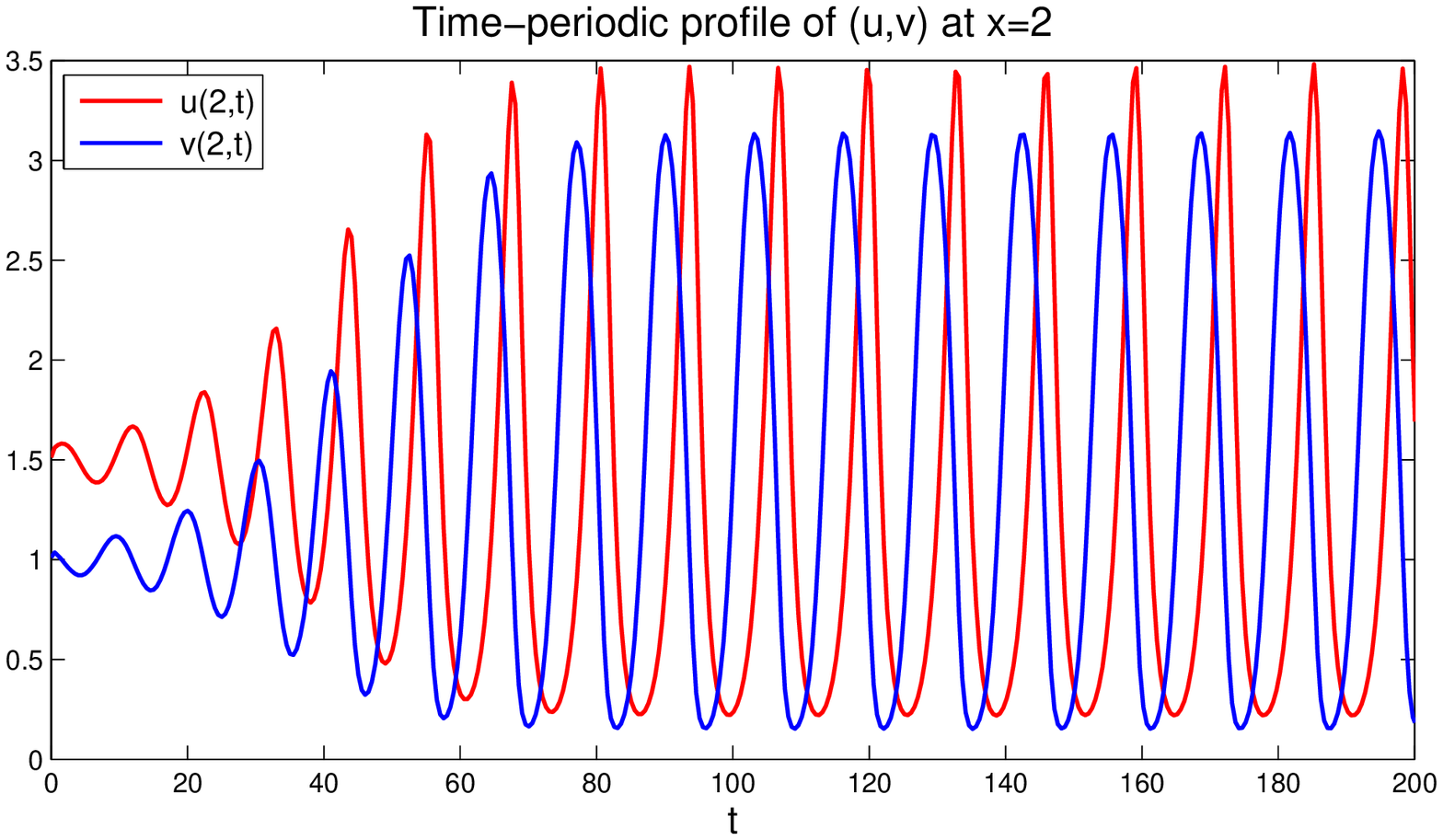}

(c)
\caption{Numerical simulation of spatially homogeneous time-periodic patterns generated by (\ref{LV-2}) with $\chi(v)=-d'(v)$ in the interval $[0, 8 \pi]$, where $d(v)=d_1(v)$ given in \eqref{cd} and parameter values are: $K=4,\gamma=2,\theta=1,\lambda=1,\mu=1, D=1/10$. The initial datum $(u_0, v_0)$ is set as a small random perturbation of the {\color{black} homogeneous} coexistence steady state $(3/2,1)$. The simulation illustrates a spatially homogeneous time-periodic coexistence patterns for the predator and the prey.}
\label{fig1}

\end{figure}

\subsection{Spatio-temporal patterns}
In this subsection, we shall present some examples to illustrate the periodic and steady state patterns. As discussed in the previous section, the Lotka-Volterra type predator-prey system (\ref{LV}) does not generate any spatially inhomogeneous patterns while the Rosenzweig-MacArthur type predator-prey system (\ref{LV-2}) with $\chi(v)=-d'(v)$ may generate periodic or steady state patterns in appropriate parameter regimes. Therefore we only consider the  Rosenzweig-MacArthur type predator-prey system (\ref{LV-2}) with $\chi(v)=-d'(v)$. We fix the value of the parameters in all simulations as follows:
\begin{equation}\label{cp}
K=4,\gamma=2,\theta=1,\lambda=1,\mu=1.
\end{equation}
Then it can be checked from \eqref{coess} that the coexistence steady state $(u_*,v_*)=(3/2,1)$. Furthermore it can be verified from \eqref{beta} that
\begin{equation*}\label{para-2}
\beta_1=\frac{1}{8}, \ \beta_3=\frac{3}{8}
\end{equation*}
where the value of $\beta_2$ depend on the specific form of $d(v)$. In this paper, we shall test three motility function $d(v)$ as follows
\begin{equation}\label{cd}
d_1(v)=\frac{1}{1+e^{2(v-1)}}, \ \ \ d_2(v)=\frac{1}{1+e^{\frac{1}{10}(v-1)}}, \ \ \ d_3(v)= \frac{1}{9+e^{2(v-1)}}
\end{equation}
Hence $d_1(v_*)=d_2(v_*)=1/2$, $d_3(v_*)=\frac{1}{10}$ and $\chi_1(v_*)=-d_1'(v_*)=\frac{1}{2}, \chi_2(v_*)=-d_2'(v_*)=\frac{1}{40}$,  $\chi_3(v_*)=-d_3'(v_*)=\frac{1}{50}$.
Next we shall numerically explore the possible patterns for different choices of $d(v)$ given in \eqref{cd}.

\begin{figure}[t]
\centering
\includegraphics[width=7cm]{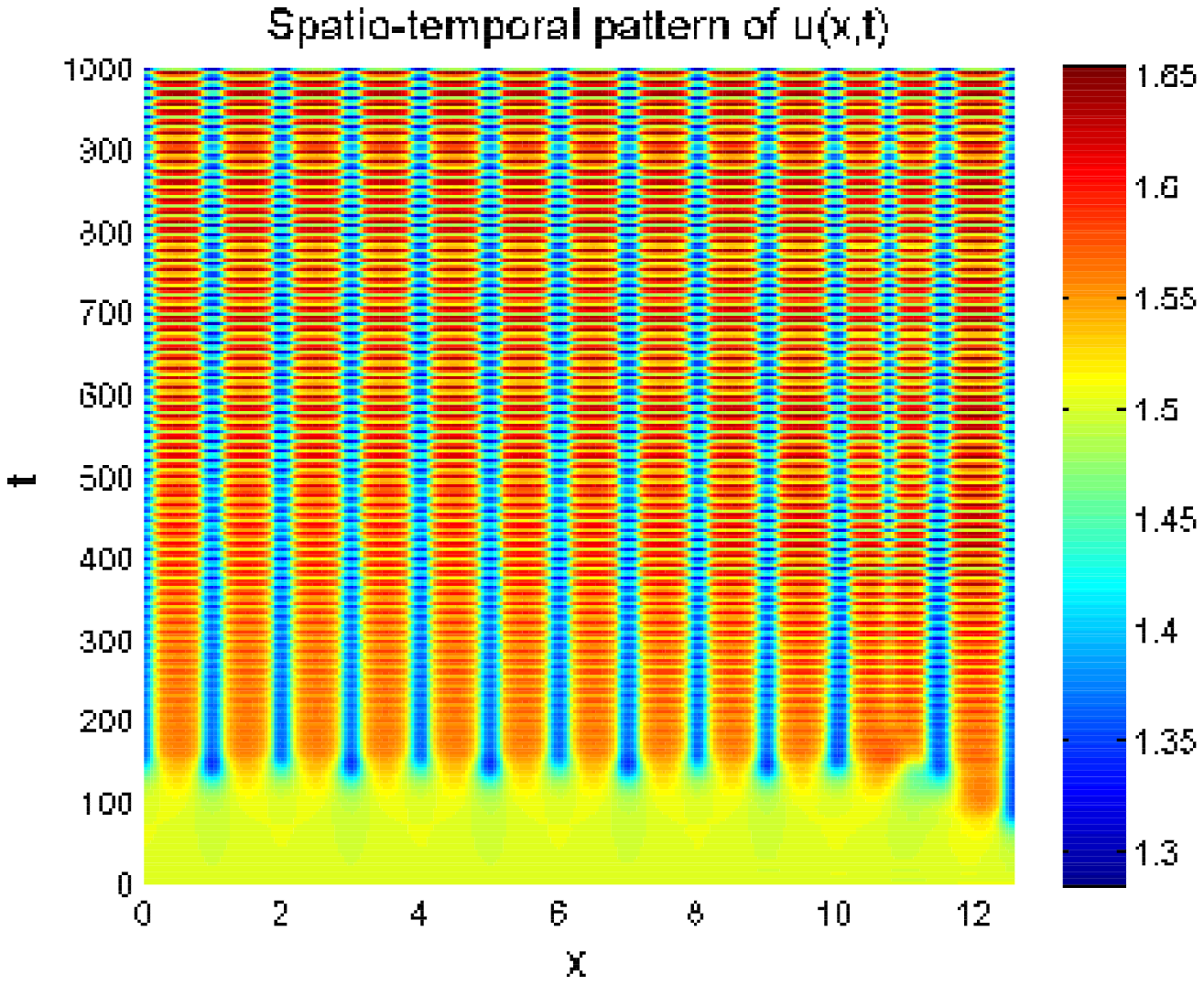}
\includegraphics[width=7cm]{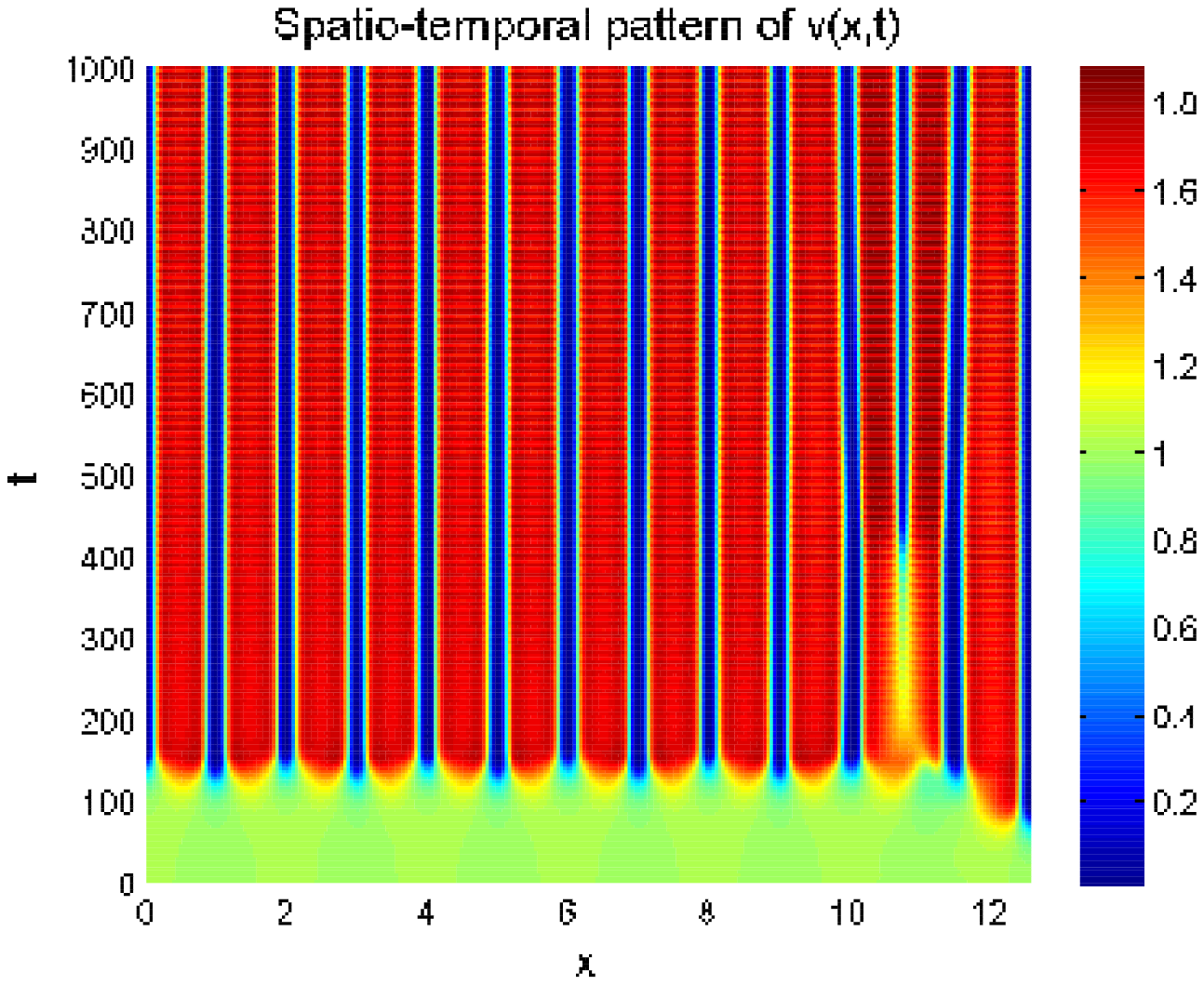}

(a) \hspace{6.5cm} (b)

\vspace{0.5cm}

\includegraphics[width=7cm]{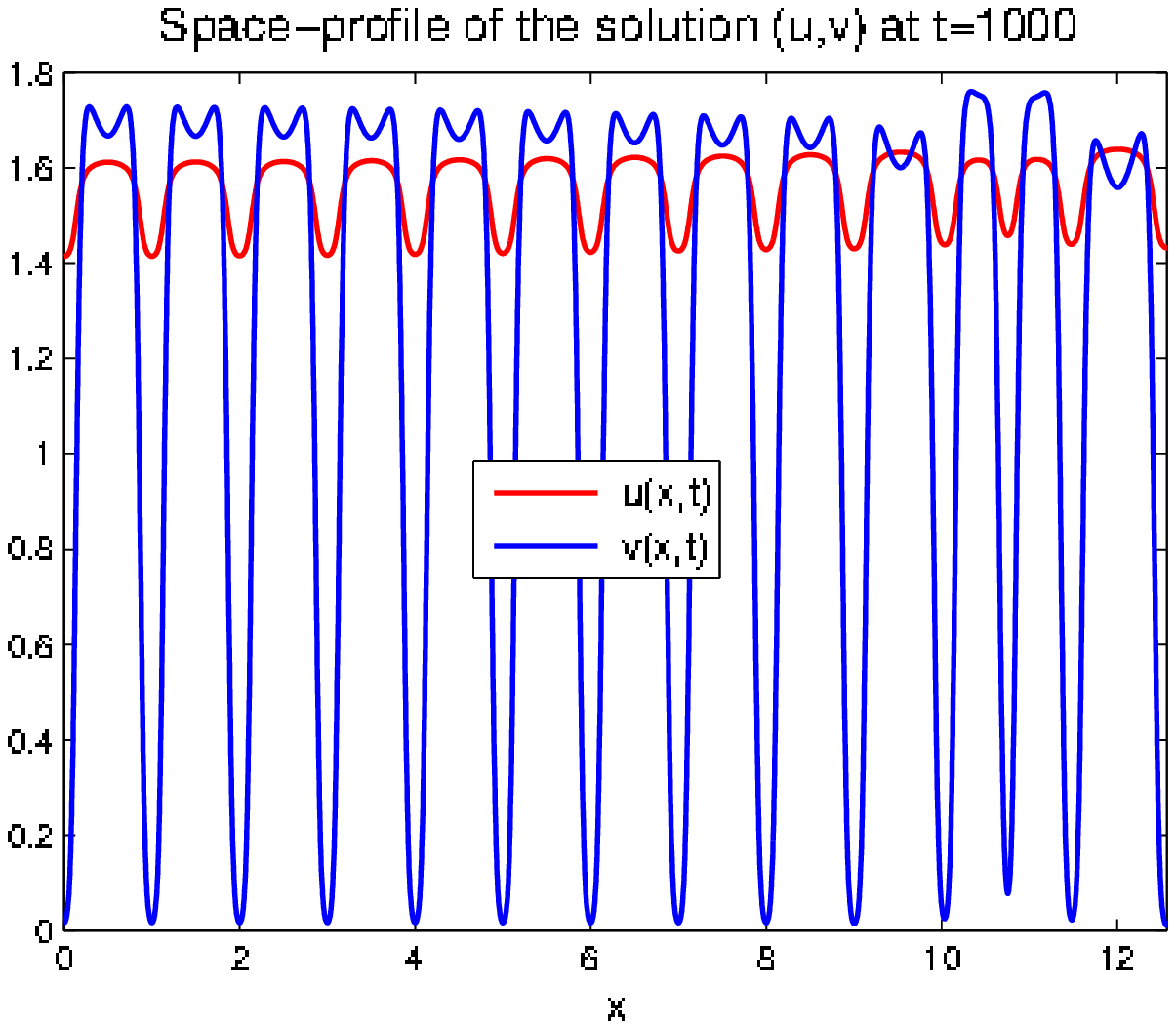}
\includegraphics[width=7cm]{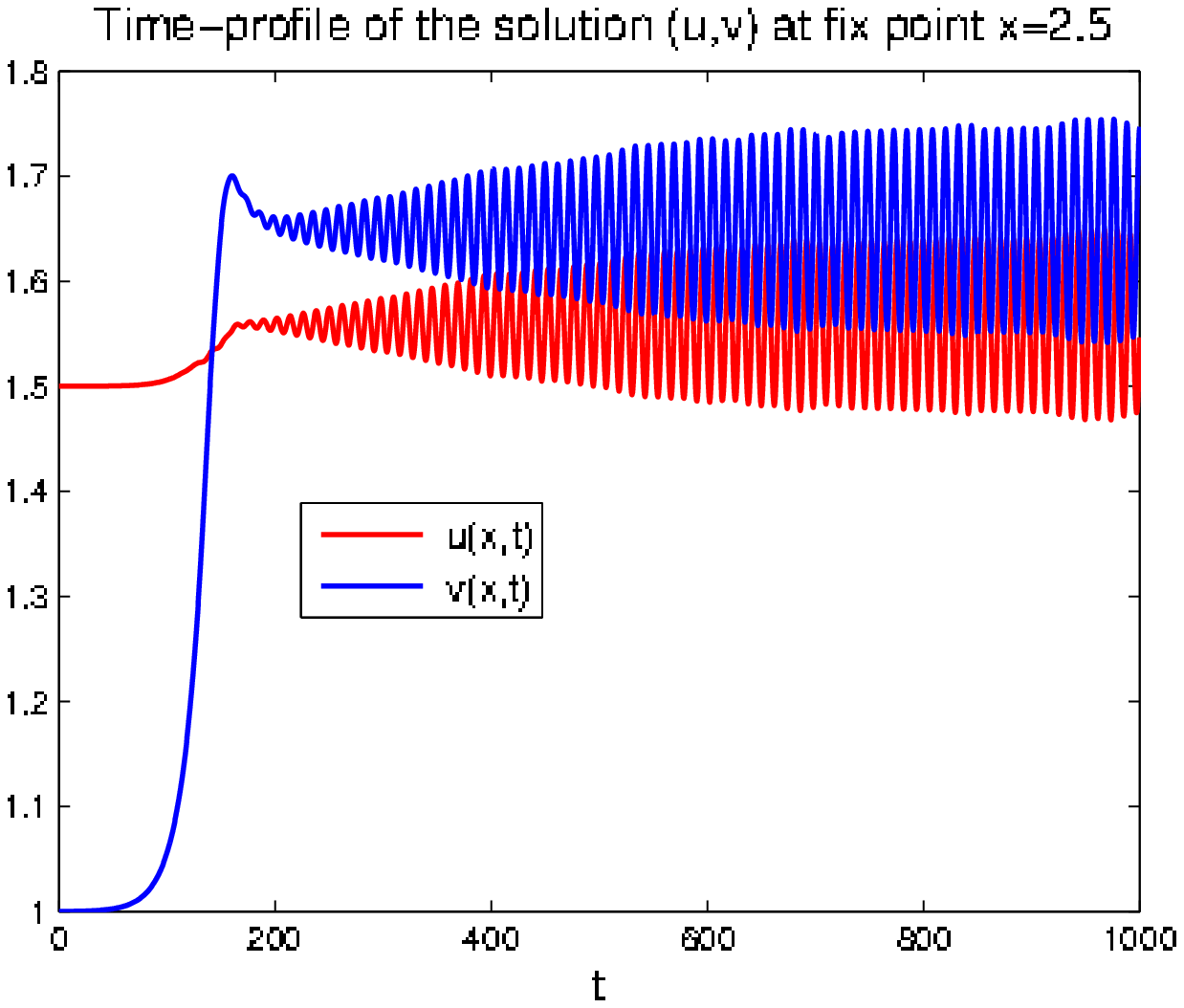}

(c) \hspace{6.5cm} (d)
\caption{Numerical simulation of spatio-temporal patterns generated by (\ref{LV-2}) with $\chi(v)=-d'(v)$ in the interval $[0, 4 \pi]$, where $d(v)=d_2(v)$ given in \eqref{cd} and parameter values are: $K=4,\gamma=2,\theta=1,\lambda=1,\mu=1, D=1/4800$. The initial datum $(u_0, v_0)$ is set as a small random perturbation of the homogeneous coexistence steady state $(3/2,1)$. }
\label{fig2}
\end{figure}

{\bf Case 1}: $d(v)=d_1(v)$. In this case, under the parameters chosen  in  \eqref{cp}, one can verify that
\begin{equation}\label{hbc}
a(D,k^2)=\left(\frac{1}{2}+D\right)k^2-\frac{1}{8}.
\end{equation}
One also can verify from (\ref{coee}) that $b(D,k^2)=\frac{D}{2}k^4+\frac{5}{16}k^2+\frac{3}{8}>0$, $\beta_2=-\frac{5}{16}$ and
$$\Delta=|a(D,k^2)|^2-4b(D,k^2)=\Big(\frac{1}{2}-D\Big)^2k^4-\frac{1}{4}\Big(\frac{11}{2}+D\Big)k^2-\frac{95}{64}.$$
This indicates that as long as $D$ is close to $1/2$, then $\Delta<0$ and Hopf bifurcation will certainly arise. One is concerned whether the steady state bifurcation will occur in this case. Indeed it can be readily checked $a(D, k^2)<0 \ \text{and}\ |a(D,k^2)|^2-4b(D,k^2)>0$ can not be fulfilled simultaneously. Hence from (\ref{SSB}), we know that the steady state bifurcation is impossible in this case. However the Hopf bifurcation will develop if $D$ is suitably chosen so that $\Delta<0$ for some $k$. For simulation, we choose $D=1/10$ such that $\Delta<0$ and $a(D,k^2)<0$ with allowable wavenumber satisfying
$
k^2<\frac{5}{24}
$
which, under the facts $k^2=(\frac{n\pi}{\ell})^2, n=0,1,2, \cdots$ with $\ell=8\pi$, gives $n=0,1,2,3$. The numerical simulations of patterns are then shown in Fig.\ref{fig1}(a)-(b) where we observe the spatially homogeneous time-periodic patterns. In principle there will be three spatial modes arising from the homogeneous coexistence steady state $(3/2,1)$, but we do not obverse the spatial inhomogeneity. This implies from the plot in Fig.\ref{fig1}(c) that as the solution amplitude become large as time increases, the nonlinearity will play a dominant role and the linearized dynamics is insufficient to explain the nonlinear behavior.

{\bf Case 2}: $d(v)=d_2(v)$. For this case, $\beta_2=\frac{7}{160}>0$ and $a(D,k^2)$ is still given by (\ref{hbc}). Hence the condition (\ref{ss-3}) is verified {\color{black} and
\begin{eqnarray*}
\Lambda=\beta_2^2-4\beta_3 D d(v_*)=\left(\frac{7}{160}\right)^2-\frac{3}{4}D.
\end{eqnarray*}
Clearly $\Lambda>0$ if $D<\frac{49}{19200}$ and $\Lambda \leq 0$ if $D\geq\frac{49}{19200}$}, which indicates from (\ref{ss-4}) that the steady state bifurcation will occur if $0<D<\frac{49}{19200}$. This is confirmed by numerical simulations shown in Fig. \ref{fig2} where we take $D=\frac{1}{4800}$ and observe the development of spatially inhomogeneous stationary patterns (see Fig. \ref{fig2} (a)-(b)). Furthermore  both the predator and the prey reach a perfect inhomogeneous coexistence state in space (see Fig.\ref{fig2}(c)) but remain oscillations in time (see Fig.\ref{fig2}(d)). It has been proved that if $d(v)$ is constant, the diffusive Rosenzweig-MacArthur predator-prey system \eqref{LV-2} will not admit spatial patterns (cf.  \cite{WWS-M3AS-2018, YWS}).  {\color{black} The spatially inhomogeneous stationary patterns shown in Fig.\ref{fig2} implies that density-dependent nonlinear motility (i.e., function $d(v)$), which leads to a cross-diffusion motion, is a trigger for pattern formation. This is a new observation although it is not justified in the paper.  When $d(v)$ is constant, the spatial patterns and time-periodic patterns have been obtained for preytaxis systems with different predator-prey interactions or mobility coefficient $\chi(v)$, see \cite{WWY-DCDS-A, WSS-JNS-2017}. We also refer to  \cite{MK-JMB-1980, JKT-EJAM-2017} for some other types of cross-diffusion which cause the emergence of spatial patterns.}
\begin{figure}[t]
\centering
\includegraphics[width=7cm]{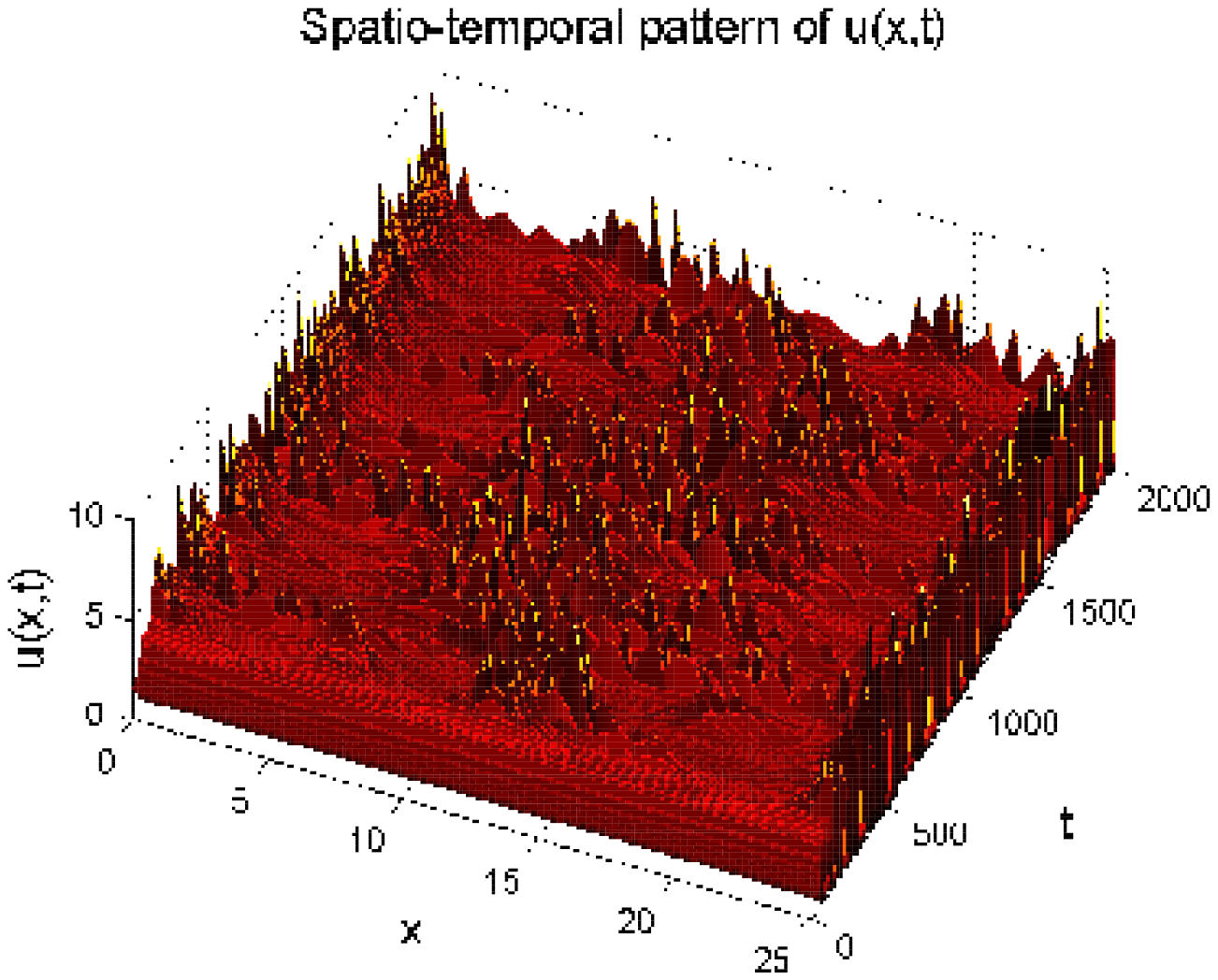}
\includegraphics[width=7cm]{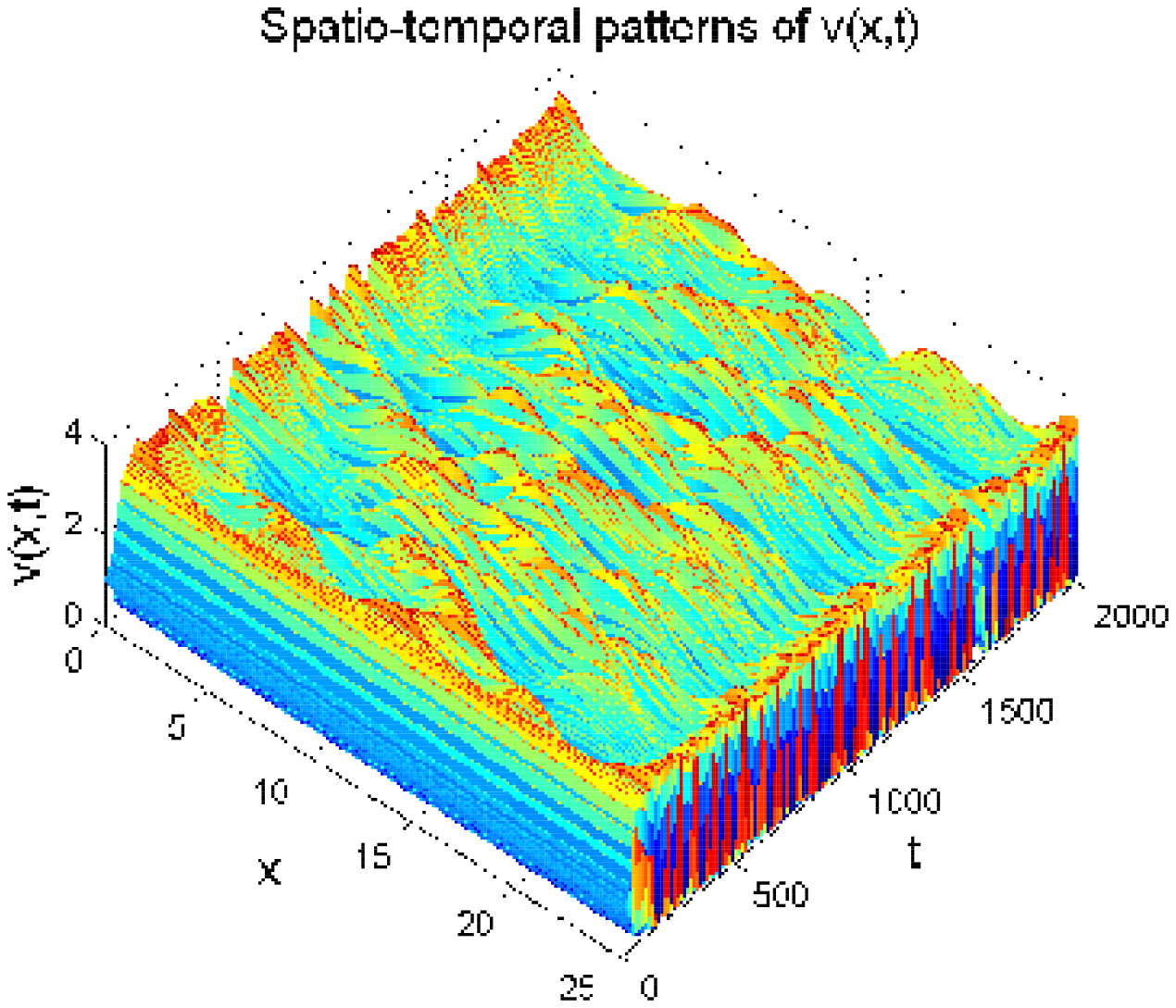}
%
%
\includegraphics[width=15cm]{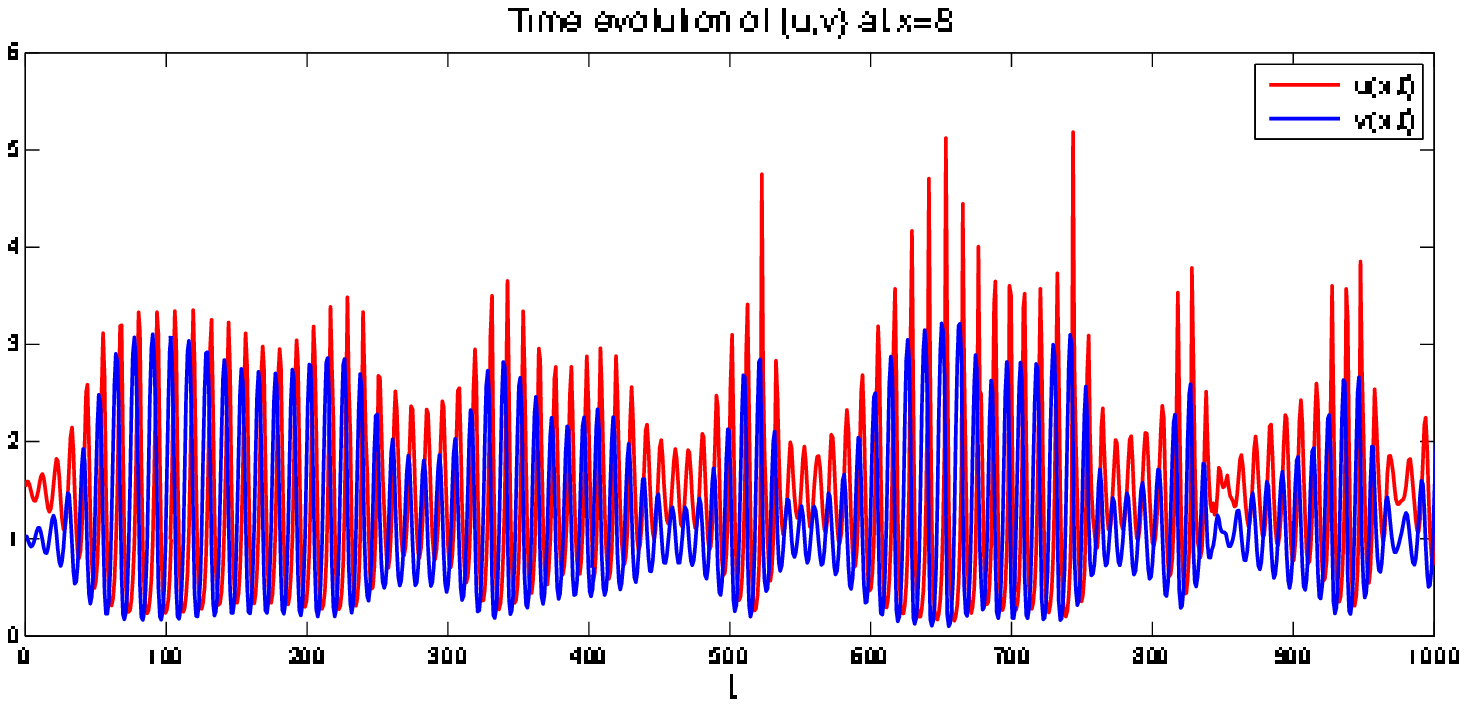}

\caption{Numerical simulation of spatio-temporal patterns generated by (\ref{LV-2}) with $\chi(v)=-d'(v)$ in the interval $[0, 8 \pi]$, where $d(v)=d_3(v)$ given in \eqref{cd} and parameter values are: $K=4,\gamma=2,\theta=1,\lambda=1,\mu=1, D=1/10$. The initial datum $(u_0, v_0)$ is set as a small random perturbation of the coexistence steady state $(3/2,1)$. }
\label{fig3}
\end{figure}

{\bf Case 3}: $d(v)=d_3(v)$. In this case, one has $d(v_*)=\frac{1}{10}$ and $\chi(v_*)=\frac{1}{50}$. Furthermore
\begin{equation*}
a(D,k^2)=(\frac{1}{10}+D)k^2-\frac{1}{8}, \ b(D,k^2)=\frac{D}{10}k^4+\frac{1}{400}k^2+\frac{3}{8}
\end{equation*}
and hence
\begin{equation*}
\Delta=|a(D,k^2)|^2-4b(D,k^2)=\left(\frac{1}{10}-D\right)^2k^4-\frac{1}{4}\left(\frac{7}{50}+D\right)k^2-\frac{95}{64}.
\end{equation*}
Choosing $D=\frac{1}{10}$, then $\Delta<0$  and $a(D,k^2)<0$ with allowable wavenumber  $k^2<\frac{5}{8}$.
Hence allowable wave modes are {\color{black}$n=0,1,2,3,4,5,6$} by noticing that $k=\frac{n}{8}$ and Hopf bifurcation (with positive real part in the temporal eigenvalue) will arise. We show the numerical simulation in Fig.\ref{fig3}, where we observe the development of chaotic spatio-temporal patterns, which are different from the patterns shown in Fig.\ref{fig1} and Fig.\ref{fig2}. They are not the periodic patterns either (see the lower panel of Fig.\ref{fig3}) as we expect from the linear stability analysis, which indicates again that the dynamics between nonlinear and linearized systems are quite different.
We also note that the simulations in Fig.\ref{fig1} and Fig.\ref{fig3} demonstrate that the Hopf bifurcation arising from the time-periodic orbits can develop into spatially homogeneous time-periodic patterns (Fig.\ref{fig1}) or chaotic spatio-temporal patterns (Fig.\ref{fig3}). The difference in the simulations shown in Fig.\ref{fig1} and Fig.\ref{fig3} lies in the choice of motility function $d(v)$. This observation hints us that the motility function $d(v)$ of the predator plays an important role in determining the spatial distribution of the predator and the prey. In particular the random motion ($d(v)$ is constant) and nonrandom motion ($d(v)$ is non-constant) will result in different patterns (i.e. spatial distribution of the predator and the prey). Hence how does the motility function $d(v)$ affects the dynamics of nonlinear predator-prey systems launches an interesting question for the future.

\bigbreak

\noindent \textbf{Acknowledgment.} {\color{black}The authors are grateful to the three referees and  one editor for their valuable comments, which greatly improved the exposition of our paper.}
 The research of H.Y. Jin was supported  by  the NSF of China No. 11871226 and  the Fundamental Research Funds for the Central Universities. The research of Z. Wang was supported by the Hong Kong RGC GRF grant No. PolyU 153298/16P.

\end{document}